\documentclass[10pt]{amsart}
\usepackage{amsfonts,amssymb,amscd,amsmath,enumerate,verbatim,calc}
\usepackage{mathrsfs}
\numberwithin{equation}{section}
\newtheorem{theorem}{Theorem}[section]
\newtheorem{corollary}[theorem]{Corollary}
\newtheorem{lemma}[theorem]{Lemma}
\newtheorem{proposition}[theorem]{Proposition}
\newtheorem{definition}[theorem]{Definition}
\newtheorem{
ples}[theorem]{Examples}
\newtheorem{example}[theorem]{Example}
\newtheorem{remark}[theorem]{Remark}

 \DeclareMathOperator{\End}{End}

\begin{document}

\title[ para-K\"{a}hler hom-Lie algebras]
 { para-K\"{a}hler hom-Lie algebras}

\bibliographystyle{amsplain}

\author[Esmaeil Peyghan and Leila Nourmohammadifar]{ E. Peyghan and L. Nourmohammadifar}
\address{Department of Mathematics, Faculty of Science, Arak University,
Arak, 38156-8-8349, Iran.}
\email{e-peyghan@araku.ac.ir,\ l-nourmohammadifar@phd.araku.ac.ir}





\begin{abstract}
In this paper, we introduce the notions of pseudo-Riemannian, para-Hermitian and para-K\"{a}hler structures on hom-Lie
algebras. In addition, we present  the characterization of these structures. Also, we provide  an example 
including  these structures. We then  introduce the phase space of a hom-Lie algebra and using the hom-left
symmetric product, we show that a para-K\"{a}hler hom-Lie algebra gives a phase space and conversely, we
can construct a para-K\"{a}hler hom-Lie algebra using a phase space.

\end{abstract}

\maketitle


\section{Introduction}
In the study of deformations of the Witt and the Virasoro algebras, Hartwig, Larsson, and Silvestrov introduced the notion of hom-Lie algebras \cite{HLS}. Indeed, some $q$-deformations of the
Witt and the Virasoro algebras have the structure of a hom-Lie algebra \cite{HLS, H}. This algebraic structure has important role among some mathematicians, because of close
relation to discrete and deformed vector fields and differential calculus \cite{HLS, LS}. In particular, hom-Lie algebras on semisimple Lie
algebras, omni-hom-Lie algebras and quadratic hom-Lie algebras have been studied in \cite{JL}, \cite{SX} and \cite{BM}, respectively. Also, authors have studied cohomology and
homology theories in \cite{AEM, CS, Y}, representation theory in \cite{S} and hom-Lie Bialgebras in \cite{Y1}.

An almost product structure on a manifold $M$ is a field $K$ of involutive
endomorphisms, i.e., $K^2=Id_{TM}$. Moreover, if the eigendistributions $T^{\pm}M$ with eigenvalues $\pm 1$ have the same constant rank, then $K$ is called almost para-complex structure. In this case the dimension of $M$ must to be even. An almost para-Hermitian structure is a pair $(K,g)$ such that $g(\cdot, \cdot)=-g(K\cdot, K\cdot)$ of an almost para-complex structure $K$ and a pseudo-Riemannian metric $g$. A manifold $M$ is called almost para-Hermitian manifold if
it is endowed with an almost para-Hermitian structure $(K,g)$. Almost para-Hermitian geometry is a topic with many analogies with the almost Hermitian geometry and also with differences \cite{IZ}. An almost para-Hermitian manifold $(M, K, g)$ is called para-K\"{a}hler, if its its Levi-Civita connection $\nabla$ satisfies $\nabla K=0$.

Recently, many researchers have studied some geometric concepts over Lie groups and Lie algebras such as complex, para-complex  and contact structures \cite{A, ABD, ABD1, CF, DV}. Inspired by these papers, this idea appeared in our mind to introduce some geometric concepts on hom-Lie algebras. 

The structure of this paper is organized as follows: In Section 2, we present the definition of hom-algebra, hom-left-symmetric algebra, hom-Lie algebra. Then we introduce representations for hom-Lie algebras and we study some properties of them. Finally, we introduce symplectic hom-Lie algebra and pseudo-Riemannian hom-algebra and we prove the existence of a unique product on this algebra, namely hom-Levi-Civita product. In Section 3, we introduce para-complex, para-Hermitian and para-K\"{a}hler structures on hom-Lie algebras. Also, we present an example of these structures. Then we study some properties of para-K\"{a}hler hom-Lie algebras. In Section 4, we define the phase space of a hom-Lie algebra and then we show that there exists a phase space on a para-K\"{a}hler home Lie algebra. Also, we prove that using phase spaces we can construct para-K\"{a}hler home Lie algebras.





\section{Representation and pseudo-Riemannian metric on hom-Lie algebras}
In this section, we present symplectic hom-Lie algebra, pseudo-Riemannian hom-algebra and hom-Levi-Civita product.
\begin{definition}
A hom-algebra is a triple $(V,\cdot , \phi)$  consisting of a linear space $V$, a bilinear map (product)
$\cdot : V \times V \rightarrow V$ and an algebra morphism $\phi : V \rightarrow V$.
\end{definition}
Let $(V, \cdot, {\phi})$ be a hom-algebra. Then for any $u \in V$, we consider maps $L_u, R_u : V \rightarrow V$ as the left and the right
multiplication by u given by $L_u(v) = u\cdot v$ and $ R_u(v) = v\cdot u$, respectively. The ‎‎\textit{commutator}
 on $V$ is given by $[u,v]=u\cdot v-v\cdot u$.
 \begin{definition}
  Let $(V, \cdot, \phi)$ be a hom-Lie algebra and $[\cdot,\cdot]$ be its commutator. We define the tensor curvature $\mathcal{K}$ of $V$ as follows
 \begin{equation}\label{SS8}
 \mathcal{K}(u,v):=L_{\phi(u)}\circ L_v-L_{ \phi(v)}\circ L_u-L_{[u,v]}\circ {\phi},
 \end{equation}
  for any $u, v\in V$.
 \end{definition}
 \begin{lemma}\label{IMPORT}
 We have 
\begin{equation}\label{SS6}
\circlearrowleft_{u,v,w}[{ \phi}(u),[v,w]]=\circlearrowleft_{u,v,w} \mathcal{K}(u,v)w,
\end{equation}
where $\circlearrowleft_{u,v,w}$ denotes cyclic sum on $u,v,w$. This equation is called hom-Bianchi identity.
 \end{lemma}
 \begin{proof}
If we consider $u,v,w\in V$, then we obtain
\begin{align*}
&\circlearrowleft_{u,v,w}[{ \phi‎‎_\mathfrak{g}}(u),[v,w]]=[{ \phi‎‎_\mathfrak{g}}(u),[v,w]]+[{ \phi‎‎_\mathfrak{g}}(v),[w,u]]+[{ \phi‎‎_\mathfrak{g}}(w),[u,v]]\\
=&L_{{ \phi‎‎_\mathfrak{g}}(u)}[v,w]-L_{[v,w]}{ \phi‎‎_\mathfrak{g}}(u)+L_{{ \phi‎‎_\mathfrak{g}}(v)}[w,u]-L_{[w,u]}{ \phi‎‎_\mathfrak{g}}(v)+L_{{ \phi‎‎_\mathfrak{g}}(w)}[u,v]-L_{[u,v]}{ \phi‎‎_\mathfrak{g}}(w)\\
=&L_{{ \phi‎‎_\mathfrak{g}}(u)}L_vw-L_{{ \phi‎‎_\mathfrak{g}}(u)}L_wv-L_{[v,w]}{ \phi‎‎_\mathfrak{g}}(u)+L_{{ \phi‎‎_\mathfrak{g}}(v)}L_wu-L_{{ \phi‎‎_\mathfrak{g}}(v)}L_uw-L_{[w,u]}{ \phi‎‎_\mathfrak{g}}(v)\\
&+L_{{ \phi‎‎_\mathfrak{g}}(w)}L_uv-L_{{ \phi‎‎_\mathfrak{g}}(w)}L_vu-L_{[u,v]}{ \phi‎‎_\mathfrak{g}}(w)
=\mathcal{K}(u,v)w+\mathcal{K}(v,w)u+\mathcal{K}(w,u)v.
\end{align*}
\end{proof}
\begin{definition}
The hom-algebra $(V,\cdot, \phi)$ is said to be a hom-associative algebra if for any $u,v,w\in V$ we have
\[
\phi(u)\cdot(v\cdot w)=(u\cdot v)\cdot\phi(w).
\]
\end{definition}
\begin{definition}
A hom-left-symmetric algebra is a hom-algebra $(V, \cdot,  \phi )$ such that
\begin{equation*}
ass_\phi(u,v,w)=ass_\phi(v,u,w),
\end{equation*}
where
\begin{align*}
ass_\phi(u,v,w)=(u\cdot v)\cdot\phi(w)-\phi(u)\cdot(v\cdot w),
\end{align*}
for any $u,v,w\in V$.
\end{definition}
\begin{remark}
Each hom-associative algebra is a hom-left-symmetric algebra with $ass_\phi(u,v,w)=0$, but the converse is not true. 
\end{remark}
\begin{definition}\cite{SB}
A hom-Lie algebra is a triple $(\mathfrak{g}, [\cdot  , \cdot ],{ \phi‎‎_\mathfrak{g}})$ consisting of a linear space $\mathfrak{g}$, a
 bilinear map (bracket) $[\cdot , \cdot ]: \mathfrak{g}\times\mathfrak{g}\rightarrow\mathfrak{g}$ and an algebra morphism ${ \phi‎‎_\mathfrak{g}}:\mathfrak{g}\rightarrow \mathfrak{g}$ satisfying
\begin{align*}
[u,v]=-[v,u],
\end{align*}
\begin{equation}\label{AM502}
\circlearrowleft_{u,v,w}[\phi(u),[v,w]]=0,
\end{equation}
for any $u,v,w\in\mathfrak{g}$. The second equation is called hom-Jacobi identity. The hom-Lie algebra $(\mathfrak{g}, [\cdot  ,\cdot ], { \phi_\mathfrak{g}})$ is called regular (involutive), if ${ \phi_\mathfrak{g}}$ is non-degenerate (satisfies ${ \phi_\mathfrak{g}}^2=1$).
\end{definition}
\begin{definition}
 A hom-algebra $(V, \cdot, \phi)$ is called
hom-Lie-admissible algebra
if its commutator bracket 
satisfies the hom-Jacobi identity.
\end{definition}
Lemma \ref{IMPORT} and the above definition, imply the following
\begin{corollary}
For a hom-Lie-admissible algebra we have $\circlearrowleft_{u,v,w} \mathcal{K}(u,v)w=0$.
\end{corollary}
\begin{proposition}\label{3.9}
If $(V, \cdot, \phi)$ is a hom-Lie-admissible algebra, then $(V, [\cdot ,\cdot ], \phi)$ is a hom-Lie algebra, where $[\cdot ,\cdot ]$ is the commutator bracket. 
\end{proposition}
\begin{proof}
Since $\phi(u\cdot v)=\phi(u)\cdot\phi(v)$ for any $u,v\in V$, it follows that
\[
\phi[u, v]=\phi(u\cdot v-v\cdot u)=\phi(u)\cdot\phi(v)-\phi(v)\cdot\phi(u)=[\phi(u), \phi(v)], 
\]
i.e., $\phi$ is algebra homomorphism with respect to commutator bracket. As hom-Jacobi identity is hold, then $(V, [\cdot ,\cdot ], \phi)$ is a hom-Lie algebra.
\end{proof}
Easily we can infer the following
\begin{proposition}\label{3.10}
A hom-left-symmetric algebra is a hom-Lie-admissible algebra.
\end{proposition}
Let $(g, [\cdot  , \cdot ], {\phi_\mathfrak{g}})$ be a hom-Lie algebra. A \textit{representation} of $\mathfrak{g}$ is a triple $(V, A, \rho)$ where $V$ is a vector space, $A\in gl(V)$ and 
$\rho: \mathfrak{g}\rightarrow gl(V)$ is a linear map satisfying
\begin{equation}\label{SS4}
\left\{
\begin{array}{cc}
&\hspace{-4cm}\rho(\phi_{\mathfrak{g}}(u))\circ A=A\circ \rho(u),\\
&\hspace{-.2cm}\rho([u,v]_{\mathfrak{g}})\circ A=\rho(\phi_{\mathfrak{g}}(u))\circ \rho(v)-\rho(\phi_{\mathfrak{g}}(v))\circ \rho(u),
\end{array}
\right.
\end{equation}
for any $u,v \in \mathfrak{g}$. If we consider $V^*$ as the dual vector
space of $V$, then we can define a linear map $\widetilde{\rho}:\mathfrak{g}\rightarrow gl(V^*)$ by 
\[
\prec\widetilde{\rho}(u)(\alpha),v\succ=-\prec({\rho}(u))^t(\alpha),v\succ=-\prec\alpha,\rho(u)(v)\succ,
\]
for any $u\in \mathfrak{g}, v\in V,  \alpha\in {V^*}$, where $(\rho(u))^t$ is the dual endomorphism of $\rho(u)\in\End(V)$ and $\prec\widetilde{\rho}(u)(\alpha),v\succ$ is $\widetilde{\rho}(u)(\alpha)(v)$. The representation $(V,A, \rho)$ is called \textit{admissible} if $(V ^*,A^*,\widetilde{\rho})$ is also a representation of $\mathfrak{g}$. In \cite{SB}, it is shown that the representation $(V, A, {\rho})$ is admissible if and only if 
the following conditions  are satisfied  
\begin{equation}\label{AM200}
\left\{
\begin{array}{cc}
&\hspace{-3.7cm}A\circ \rho(\phi_{\mathfrak{g}}(u))=\rho(u)\circ A,\\
&A\circ\rho([u,v]_\mathfrak{g})=\rho(u)\circ \rho(\phi_{\mathfrak{g}}(v))-\rho(v)\circ\rho(\phi_{\mathfrak{g}}(u)).
\end{array}
\right.
\end{equation}
Let $(g, [\cdot  , \cdot ], {\phi_\mathfrak{g}})$ be an hom-Lie algebra and $ad : \mathfrak{g}\rightarrow End(\mathfrak{g})$ be an operator defined for any
$u,v\in \mathfrak{g}$ by $ad(u)(v) = [u, v]$.  We obtain
\[
ad(\phi_\mathfrak{g}(u))(\phi_\mathfrak{g}(v))=[\phi_\mathfrak{g}(u), \phi_\mathfrak{g}(v)]=\phi_\mathfrak{g}([u, v])=\phi_\mathfrak{g}(ad(u)(v)).
\]
Moreover, it is easy to see that  
\[
ad{[u,v]}\circ{ \phi_\mathfrak{g}}=ad({ \phi_\mathfrak{g}}(u))\circ ad(v)-ad({ \phi_\mathfrak{g}}(v))ad(u).
\]
Thus $(\mathfrak{g}, ad, { \phi_\mathfrak{g}})$ is a representation of $\mathfrak{g}$, which is called an \textit{adjoint representation of $\mathfrak{g}$} (see \cite{SB}, for more details).
\begin{lemma}
Let $(g, [\cdot  , \cdot ], {\phi_\mathfrak{g}})$ be an involutive hom-Lie algebra. Then the adjoint representation  $(\mathfrak{g}, ad, { \phi_\mathfrak{g}})$ is
admissible.
\end{lemma} 
\begin{proof}
In \cite{SB}, it is shown that $(\mathfrak{g}, ad, { \phi_\mathfrak{g}})$ is admissible if and only if 
\begin{align*}
[(Id_\mathfrak{g}-\phi^2_\mathfrak{g})(u), \phi_\mathfrak{g}(v)]&=0,\\
[(Id_\mathfrak{g}-\phi^2_\mathfrak{g})(u), [\phi_\mathfrak{g}(v), w]]&=[(Id_\mathfrak{g}-\phi^2_\mathfrak{g})(v), [\phi_\mathfrak{g}(u), w]].
\end{align*}
Since $\phi^2_\mathfrak{g}=id_\mathfrak{g}$, the above equations are hold and so $(\mathfrak{g}, ad, { \phi_\mathfrak{g}})$ is
admissible.
\end{proof}
\begin{proposition}
Let $(\mathfrak{g}, [\cdot  ,\cdot ]_{\mathfrak{g}},{ \phi_\mathfrak{g}})$ be a hom-Lie algebra, $\mathfrak{g}^*$ be its dual space and $(\mathfrak{g}, A, \rho)$ be an admissible representation of $\mathfrak{g}$. Then on the direct sum $\mathfrak{g}\oplus\mathfrak{g}^*$, there is a hom-Lie algebra structure given as follows 
\begin{align*}
[u+\alpha,v+\beta]&:=[u,v]_{\mathfrak{g}}+\widetilde{\rho}(u)(\beta)-\widetilde{\rho}(v)(\alpha),\\
\Phi(u+\alpha)&:=\phi_{\mathfrak{g}}(u)+A^*(\alpha),
\end{align*}
for any $u,v\in \mathfrak{g}$ and $\alpha,\beta\in\mathfrak{g}^*$.
\end{proposition}
\begin{proof}
 Let $u,v\in \mathfrak{g}, \alpha,\beta\in\mathfrak{g}^*$. Obviously, we have 
\[
[u+\alpha,v+\beta]=-[v+\beta,u+\alpha].
\]
The hom-Jacobi identity is satisfied, because
\begin{align*}
&\circlearrowleft‎‎‎_{(u,\alpha),(v,\beta),(w,\gamma)}[\Phi(u+\alpha),[v+\beta,w+\gamma]]=\circlearrowleft‎‎‎_{(u,\alpha),(v,\beta),(w,\gamma)}[\phi_{\mathfrak{g}}(u)+A^*(\alpha),[v,w]_{\mathfrak{g}}+\widetilde{\rho}(v)(\gamma)-\widetilde{\rho}(w)(\beta)]\\
&=\circlearrowleft‎‎‎_{(u,\alpha),(v,\beta),(w,\gamma)}\big([\phi_{\mathfrak{g}}(u),[v,w]_{\mathfrak{g}}]_{\mathfrak{g}}+\widetilde{\rho}(\phi_{\mathfrak{g}}(u))(\widetilde{\rho}(v)(\gamma)-\widetilde{\rho}(w)(\beta))-\widetilde{\rho}([v,w]_{\mathfrak{g}})(A^*(\alpha))\big).
\end{align*}
On the other hand, $(\mathfrak{g}^*, A^*, \widetilde{\rho})$ is a representation of $\mathfrak{g}$. Therefore we have 
\[
\widetilde{\rho}([v,w]_{\mathfrak{g}})\circ A^*=\widetilde{\rho}(\phi_{\mathfrak{g}}(v))\circ \widetilde{\rho}(w)-\widetilde{\rho}(\phi_{\mathfrak{g}}(w))\circ \widetilde{\rho}(v).
\]
Two above equations imply 
\begin{align*}
&\circlearrowleft‎‎‎_{(u,\alpha),(v,\beta),(w,\gamma)}[\Phi(u+\alpha),[v+\beta,w+\gamma]]
=\circlearrowleft‎‎‎_{(u,\alpha),(v,\beta),(w,\gamma)}[\phi_{\mathfrak{g}}(u),[v,w]_{\mathfrak{g}}]_{\mathfrak{g}}\\
&+\widetilde{\rho}(\phi_{\mathfrak{g}}(u))(\widetilde{\rho}(v)(\gamma)-\widetilde{\rho}(w)(\beta))+\widetilde{\rho}(\phi_{\mathfrak{g}}(v))(\widetilde{\rho}(w)(\alpha)-\widetilde{\rho}(u)(\gamma))+\widetilde{\rho}(\phi_{\mathfrak{g}}(w))(\widetilde{\rho}(u)(\beta)-\widetilde{\rho}(v)(\alpha))\\
&-\widetilde{\rho}(\phi_{\mathfrak{g}}(v))( \widetilde{\rho}(w)(\alpha))+\widetilde{\rho}(\phi_{\mathfrak{g}}(w))( \widetilde{\rho}(v)(\alpha))-\widetilde{\rho}(\phi_{\mathfrak{g}}(w))( \widetilde{\rho}(u)(\beta))+\widetilde{\rho}(\phi_{\mathfrak{g}}(u))( \widetilde{\rho}(w)(\beta))\\
&-\widetilde{\rho}(\phi_{\mathfrak{g}}(u))( \widetilde{\rho}(v)(\gamma))+\widetilde{\rho}(\phi_{\mathfrak{g}}(v))( \widetilde{\rho}(u)(\gamma))=0.
\end{align*}
\end{proof}
\begin{definition}
A symplectic hom-Lie algebra is a regular hom-Lie algebra $(\mathfrak{g}, [,],\phi_{g})$ endowed with a bilinear skew-symmetric
non-degenerate form $\omega$ which is $2$-hom-cocycle, i.e.
\begin{align*}
 \omega([u, v], { \phi_\mathfrak{g}}(w)) &+ \omega([w,u],{ \phi_\mathfrak{g}}( v)) + \omega([v,w],{ \phi_\mathfrak{g}}( u))=0,\\
&\omega(\phi(u),\phi(v))=\omega(u,v).
\end{align*}
In this case, $\omega$ is called a symplectic structure on $\mathfrak{g}$ and $(\mathfrak{g}, [\cdot, \cdot], \phi_\mathfrak{g}, \omega)$ is called symplectic hom-Lie algebra.
\end{definition}
\begin{example}\label{IMEX}
We consider a $4$-dimensional linear space $\mathfrak{g}$ with an arbitrary basis $\{e_1,e_2,e_3,e_4\}$. We define the bracket
and linear map $\phi$ on $\mathfrak{g}$ as follows 
\begin{align*}
[e_1,e_2]=&ae_1,\ \ \ [e_1,e_3]=be_4,\ \ \ \ [e_2,e_4]=a e_4,
\end{align*}
and
\[
\phi(e_1)=-e_1,\ \ \ \phi(e_2)=e_2,\ \ \ \phi(e_3)=-e_3, \ \ \phi(e_4)=e_4.
\]
The above bracket is not a Lie bracket on $\mathfrak{g}$ if $a\neq 0$ and $b\neq0$, because
\[
[e_1,[e_2,e_3]]+[e_2,[e_3,e_1]+[e_3,[e_1,e_2]]=[e_2,-be_4]+[e_3,ae_1]=-2abe_4.
\]
It is easy to see that 
\begin{align*}
[\phi(e_1),\phi(e_2)]=-ae_1=\phi([e_1, e_2]),\ \ [\phi(e_1),\phi(e_3)]=be_4=\phi([e_1, e_3]),\ \ [\phi(e_2),\phi(e_4)]=a e_4=\phi([e_2, e_4]),
\end{align*}
i.e., $\phi$ is the algebra morphism. Also, we can deduce 
\begin{equation*}
[\phi(e_i),[e_j,e_k]]+[\phi(e_j),[e_k,e_i]+[\phi(e_k),[e_i,e_j]]=0,\ \ \ i,j,k=1, 2, 3, 4.
\end{equation*}
Thus $(\mathfrak{g}, [\cdot,\cdot], \phi)$ is a hom-Lie algebra. Now we consider the bilinear skew-symmetric
non-degenerate form $\Omega$ as follows:
\begin{equation}
\begin{bmatrix}
0&0&A&0\\
0&0&0&\frac{a}{b}A\\
-A&0&0&0\\
0&-\frac{a}{b}A&0&0
\end{bmatrix}
.
\end{equation}
Then we get
\begin{align*}
\Omega(\phi(e_1),\phi(e_3))=&A=\Omega(e_1,e_3), \ \ \ \Omega(\phi(e_2),\phi(e_4))=\frac{a}{b}A=\Omega(e_2,e_4),\\
\Omega(\phi(e_1),\phi(e_2))=&0=\Omega(e_1,e_2),\ \ \ \ \  \Omega(\phi(e_1),\phi(e_4))=0=\Omega(e_1,e_4),\\
\Omega(\phi(e_2),\phi(e_3))=&0=\Omega(e_2,e_3),\ \ \ \ \  \Omega(\phi(e_3),\phi(e_4))=0=\Omega(e_3,e_4),
\end{align*}
and 
\begin{align*}
\Omega([e_i,e_j],\phi(e_k))+\Omega([e_j,e_k],\phi(e_i))+\Omega([e_k,e_i],\phi(e_j))=0,\ \ \ \ \  i,j,k=1, 2, 3, 4.
\end{align*}
The above relations show that $\Omega$ is $2$-hom-cocycle and so $(\mathfrak{g}, [\cdot,\cdot], \phi,\Omega)$ is a symplectic hom-Lie algebra.
\end{example}
\begin{theorem}
Let $(\mathfrak{g}, [\cdot, \cdot], \phi_{g}, \omega)$ be an involutive  symplectic hom-Lie algebra. Then there exists a
 hom-left-symmetric algebra structure $\bold{a}$ on $\mathfrak{g}$ satisfying
\begin{equation}\label{AM10}
\omega(\bold{a}(u, v), { \phi_\mathfrak{g}}(w)) = - \omega({ \phi_\mathfrak{g}}(v),[u, w] ),
\end{equation}
\[
\bold{a}(u, v) -\bold{a}(v, u) = [u, v],
\]
for any $u,v\in \mathfrak{g}$.
\end{theorem}
\begin{proof}
Let $\omega \in\wedge^2 ‎‎\mathfrak{g}^*$. Then we consider $\flat : ‎‎\mathfrak{g} \rightarrow ‎‎\mathfrak{g}^*$ as the isomorphism
given by $\flat(v)=\omega(v,\cdot)$ for any $v\in \mathfrak{g}$, which gives  
 $\prec \flat(u), v‎\succ=\omega(u,v)$, for any $u\in ‎‎\mathfrak{g}$. 
Since $\flat$ is a bijective and the adjoint representation of $\mathfrak{g}$ is admissible, we can define 
$
\bold{a}(u, v)= \flat^{-1} ad_{ \phi‎‎_\mathfrak{g}(u)}^*\flat(v).
$
Using the above expressions we can write 
\begin{equation*}
\omega(\bold{a}(u, v),{ \phi_\mathfrak{g}}(w))= \prec ad_{ \phi_\mathfrak{g}(u)}^* \flat(v),{ \phi_\mathfrak{g}}(w)‎\succ=-\prec\flat(v),{ \phi_\mathfrak{g}}[u,w]‎\succ=-\omega({ \phi_\mathfrak{g}}(v),[u, w] ).
\end{equation*}
The rest of the proof is similar to Proposition 5.10 in \cite{SC}.
\end{proof}
\begin{definition}
Let $ (\mathfrak{g}, [\cdot, \cdot], { \phi_\mathfrak{g}})$ be a finite-dimensional hom-Lie algebra endowed
with a bilinear symmetric non-degenerate form $< , >$ such that for any $u,v\in \mathfrak{g}$ the following equation is satisfied
\begin{equation}\label{AM302}
	<\phi_\mathfrak{g}(u), \phi_\mathfrak{g}(v)>=<u, v>.
\end{equation}
Then, we say that $\mathfrak{g}$ admits a pseudo-Riemannian metric $<,>$ and $(\mathfrak{g}, [\cdot, \cdot], \phi_\mathfrak{g}, <,>)$  is called pseudo-Riemannian hom-Lie algebra.
\end{definition}
\begin{corollary}\label{AM350}
	Let $(\mathfrak{g}, [\cdot, \cdot], \phi_\mathfrak{g}, <,>)$ be a pseudo-Riemannian involutive hom-Lie algebra. Then for any $u,v\in\mathfrak{g}$ we have 
	\begin{equation}\label{EL1}
	<\phi_\mathfrak{g}(u), v>=<u, \phi_\mathfrak{g}(v)>.
	\end{equation}
\end{corollary}
\begin{theorem}
Let $(\mathfrak{g}, [\cdot, \cdot], \phi_\mathfrak{g}, <,>)$ be a pseudo-Riemannian hom-Lie algebra such that $\phi_\mathfrak{g}$ is isomorphism. There exists the unique product $\cdot$ on it which satisfies
\begin{equation}\label{SL}
[u,v]=u\cdot v-v\cdot u,
\end{equation}
\begin{equation}\label{L2}
<u\cdot v,{ \phi_\mathfrak{g}}(w)>=-<{ \phi_\mathfrak{g}}(v),u\cdot w>.
\end{equation}
This product is called the hom-Levi-Civita.
\end{theorem}
\begin{proof}
At first, we let that this product there exists and then we prove the uniqueness of it. Using (\ref{L2}) we deduce the following
\begin{align*}
&<u\cdot v,{ \phi_\mathfrak{g}}(w)>+<{ \phi_\mathfrak{g}}(v),u\cdot w>=0,\\
&<v\cdot w,{ \phi_\mathfrak{g}}(u)>+<{ \phi_\mathfrak{g}}(w), v\cdot u>=0,\\
&-<w\cdot u,{ \phi_\mathfrak{g}}(v)>-<{ \phi_\mathfrak{g}}(u), w\cdot v>=0.
\end{align*}
Adding the above three equations give
\[
<u\cdot v+v\cdot u, \phi_\mathfrak{g}(w)>=<w\cdot u-u\cdot w, \phi_\mathfrak{g}(v)>+<w\cdot v-v\cdot w, \phi_\mathfrak{g}(u)>.
\]
Applying (\ref{SL}) in the above equation, we get 
\begin{equation}\label{Koszul}
	2<u\cdot v,{ \phi_\mathfrak{g}}(w)>=<[u,v],{ \phi_\mathfrak{g}}(w)>+<[w,v],{ \phi_\mathfrak{g}}(u)>+<[w,u],{ \phi_\mathfrak{g}}(v)>.
\end{equation}
The above formula that we call it Koszul's formula, prove the uniqueness of $\cdot$. To prove the existence of it, we consider a basis $\{e_i\}_{i=1,\cdots, n}$ for $\mathfrak{g}$. Since $e_i\cdot e_j$, $[e_i, e_j]$ and $\phi_\mathfrak{g}(e_i)$ belong to $\mathfrak{g}$, then we can write 
\[
e_i\cdot e_j=\Gamma_{ij}^ke_k,\ \ \ e_i\cdot e_j=c_{ij}^ke_k,\ \ \  \phi_\mathfrak{g}(e_i)=\phi_i^ke_k,
\]
where $\Gamma_{ij}^k, c_{ij}^k\in\mathbb{F}$ and we use the Einstein's summation notation (this notation express that if an index appears both in the above and below, then it is summation index). Setting the above equations in (\ref{Koszul}) imply
\[
2\Gamma_{ij}^h\phi_h^r<e_k, e_r>=c_{ij}^r\phi^s_k<e_r, e_s>+c_{kj}^r\phi^s_i<e_r, e_s>+c_{ki}^r\phi^s_j<e_r, e_s>.
\]
Since $<,>$ is non-degenerated, $[<,>]$ is invertible. We denote the components of $[<,>]^{-1}$ by $<e^i, e^j>$. Also, we denote by $\widetilde{\phi}^i_j$ the component of $\phi_\mathfrak{g}^{-1}$ (note that $\phi_\mathfrak{g}$ is an isomorphism). Multiplying $<e^k, e^m>$ to the sides of the above equation gives
\[
\Gamma_{ij}^h\phi_h^m=\frac{1}{2}<e^k, e^m>\{c_{ij}^r\phi^s_k<e_r, e_s>+c_{kj}^r\phi^s_i<e_r, e_s>+c_{ki}^r\phi^s_j<e_r, e_s>\}.
\]
Multiplying $\widetilde{\phi}^l_m$ to the sides of the above equation yields
\[
\Gamma_{ij}^l=\frac{1}{2}<e^k, e^m>\widetilde{\phi}^l_m\{c_{ij}^r\phi^s_k<e_r, e_s>+c_{kj}^r\phi^s_i<e_r, e_s>+c_{ki}^r\phi^s_j<e_r, e_s>\}.
\]
This equation shows that $\cdot$ there exists and it is easy to see that it satisfies in (\ref{SL}) and (\ref{L2}).
\end{proof}
\section{ para-K\"{a}hler hom-Lie algebras}
In this section, we introduce para-complex, para-Hermitian and para-K\"{a}hler structures on hom-Lie algebras and we study some properties of para-K\"{a}hler hom-Lie algebras.
\begin{definition}
	An almost product structure on an involutive hom-Lie algebra $(\mathfrak{g}, [\cdot  , \cdot ], { \phi_\mathfrak{g}})$, is an isomorphism $ K : ‎‎\mathfrak{g} \rightarrow ‎‎\mathfrak{g}$ such that
	\[
	K^2 = Id_‎‎{\mathfrak{g}},\ \ \ \ { \phi_\mathfrak{g}}\circ K=K\circ{ \phi_\mathfrak{g}}.
	\] 
Also, $(\mathfrak{g}, [\cdot  , \cdot ], { \phi_\mathfrak{g}}, K)$ is called almost product hom-Lie algebra.
\end{definition}
	The above equations deduce $({ \phi_\mathfrak{g}}\circ K)^2=Id_‎‎{\mathfrak{g}}$. Thus, we can write $\mathfrak{g}$ as $\mathfrak{g}=\mathfrak{g}^1\oplus \mathfrak{g}^{-1}$, where
	\[
	\mathfrak{g}^1:=ker({{ \phi‎‎_\mathfrak{g}}\circ K-Id_\mathfrak{g}}),\ \ \mathfrak{g}^{-1}:=ker({ \phi_\mathfrak{g}}\circ K+Id_{\mathfrak{g}} ).
	\]
	If $\mathfrak{g}^1$ and $\mathfrak{g}^{-1}$ have the same dimension $n$, then $K$ is called an \textit{almost para-complex} structure on $(\mathfrak{g}, [\cdot, \cdot], \phi_\mathfrak{g})$ (in this case the dimensional of $\mathfrak{g}$ is even). The \textit{Nijenhuis torsion} of ${{ \phi_\mathfrak{g}}\circ K}$ is defined by
\begin{equation}\label{w1}
	4N_{{ \phi_\mathfrak{g}}\circ K}(u,v)=[({{ \phi_\mathfrak{g}}\circ K})u, ({{ \phi_\mathfrak{g}}\circ K})v]-{{ \phi_\mathfrak{g}}\circ K}[({{ \phi_\mathfrak{g}}\circ K})u, v]-{{ \phi_\mathfrak{g}}\circ K}[u,({{ \phi_\mathfrak{g}}\circ K})v]+[u,v],
\end{equation}
for all $u,v\in \mathfrak{g}$. An almost product (almost para-complex) structure is called \textit{product (para-complex)} if $N_{{ \phi_\mathfrak{g}}\circ K}=0$. In the following for simplicity, we set $N=N_{{ \phi_\mathfrak{g}}\circ K}$.
\begin{definition}\label{L100}
	An almost para-hermitian structure of a hom-Lie algebra $(\mathfrak{g}, [\cdot,\cdot],\phi_{\mathfrak{g}})$ is a pair $(K, <,>)$ consisting of an almost para-complex structure
	and a pseudo-Riemannian metric $<,>$, such that for each $u,v \in \mathfrak{g}$
	\begin{equation}\label{EL}
	<({ \phi_\mathfrak{g}}\circ K)u, ({ \phi_\mathfrak{g}}\circ K)v>=- <u , v>.
	\end{equation}
Also the pair $(K, <,>)$ is called para-Hermitian structure if $N=0$. In this case, $(\mathfrak{g}, [\cdot, \cdot], \phi_\mathfrak{g}, K, <,>)$ is called para-Hermitian hom-Lie algebra.
\end{definition}
\begin{definition}
A para-K\"{a}hler hom-Lie algebra is a pseudo-Riemannian hom-Lie algebra $(\mathfrak{g}, [\cdot, \cdot], \phi_{\mathfrak{g}})$ endowed with an almost product structure $K$, such that ${ \phi_\mathfrak{g}}\circ K$ is skew-symmetric with respect to $<,>$, and ${ \phi_\mathfrak{g}}\circ K$ is invariant with respect to the Levi-Civita product, i.e., $L_u\circ { \phi_\mathfrak{g}}\circ K={ \phi_\mathfrak{g}}\circ K\circ L_u$ for any $u\in \mathfrak{g}$.
\end{definition}
Note that condition $(u.({ \phi_\mathfrak{g}}\circ K))(v)=({ \phi_\mathfrak{g}}\circ K)(u.v)$ equivalent with 
\begin{equation}\label{AM1}
({ \phi_\mathfrak{g}}\circ K)(u).({ \phi_\mathfrak{g}}\circ K)(v)=({ \phi_\mathfrak{g}}\circ K)(({ \phi_\mathfrak{g}}\circ K)(u).v),
\end{equation}
and
\begin{equation}\label{AM2}
u.v=({ \phi_\mathfrak{g}}\circ K)(u.({ \phi_\mathfrak{g}}\circ K)(v)).
\end{equation}
\begin{example}
We consider the home-Lie algebra $(\mathfrak{g}, [\cdot,\cdot], \phi,\Omega)$ introduced in Example \ref{IMEX}. We define the metric $<,>$ of $\mathfrak{g}$ as follows
\[
\begin{bmatrix}
0&0&A&0\\
0&0&0&\frac{a}{b}A\\
A&0&0&0\\
0&\frac{a}{b}A&0&0
\end{bmatrix}
.
\]
It is easy to see that  $<\phi(e_i),e_j>=0=<e_i,\phi(e_j)>$, for all $i,j=1,2,3,4$, except 
\begin{align*}
<\phi(e_1),e_3>=&-A=<e_1,\phi(e_3)>,\ \ \ \ \ <\phi(e_2),e_4>=\frac{a}{b}A=<e_2,\phi(e_4)>.
\end{align*}
Thus (\ref{SL}) is hold and so $(\mathfrak{g}, [\cdot, \cdot], \phi_{\mathfrak{g}}, <,>)$ is a pseudo-Riemannian hom-Lie algebra. If isomorphism $K$ is determined as
\[
K(e_1)=-e_1,\ \ \ K(e_2)=e_2,\ \ \ K(e_3)=e_3, \ \ K(e_4)=- e_4,
\]
then we have 
\[
K^2(e_1)=K^2(e_2)=K^2(e_3)=K^2(e_4)=Id_\mathfrak{g}.
\]
Moreover, using the above equations deduce 
\[
\phi^2(e_1)=\phi^2(e_2)=\phi^2(e_3)=\phi^2(e_4)=Id_\mathfrak{g},
\]
and
\begin{align*}
(K\circ\phi)e_1=&e_1=(\phi\circ K)e_1,\ \ \ (K\circ\phi)e_2=e_2=(\phi\circ K)e_2,\\
(K\circ\phi)e_3=&-e_3=(\phi\circ K)e_3,\ \ \ (K\circ\phi)e_4=-e_4=(\phi\circ K)e_4.\\
\end{align*}
Thus $K$ is an almost product structure on $(g, [\cdot, \cdot], \phi_{\mathfrak{g}})$. Since $\{e_1, e_2\}\in\mathfrak{g}^1$ and $\{e_3, e_4\}\in\mathfrak{g}^{-1}$, we also deduce that $K$ is an almost para-complex structure on $(\mathfrak{g}, [\cdot, \cdot], \phi_{\mathfrak{g}})$. Moreover, we obtain 
\begin{align*}
N(e_i,e_j)=0,\ \ \ \ \forall i,j=1,2,3,4,
\end{align*}
i.e., $K$ is a para-complex structure on $(\mathfrak{g}, [\cdot,\cdot], \phi)$. It is easy to check that $<({ \phi‎‎_\mathfrak{g}}\circ K)(e_i), ({ \phi‎‎_\mathfrak{g}}\circ K)(e_j)>=-<e_i, e_j>$, for all $i, j=1,2,3,4$, and so $(\mathfrak{g}, [\cdot, \cdot], \phi_{\mathfrak{g}}, <,>, K)$ is a para-Hermitian hom-Lie algebra. Now, we study para-K\"{a}hlerian property for this hom-Lie algebra. At first, we must  obtain the Levi-Civita product for it. If we denote this product with $\cdot$, then we have $e_i\cdot e_j=\sum_{k=1}^{4}A_{ij}^ke_k$, for all $i, j=1, 2, 3, 4$. From Koszul's formula given by (\ref{Koszul}), we get $<e_1\cdot e_1, \phi(e_i)>=0$, for all $i=1, 2, 3, 4$, which give $e_1\cdot e_1=0$. Again (\ref{Koszul}) gives $<e_1\cdot e_2, \phi(e_i)>=0$, for all $i=1, 2, 4$, which imply $A_{12}^2=A_{12}^3=A_{12}^4=0$. But we obtain 
\[
-2<e_1\cdot e_2, \phi(e_3)>=-a<e_1, e_3>-b<e_2, e_4>=-2aA,
\]
which yields $A_{12}^1=a$. Therefore we deduce $e_1\cdot e_2=ae_1$. This relation together $[e_1, e_2]=e_1\cdot e_2-e_2\cdot e_1$ imply $e_2\cdot e_1=0$. In a similar way, we can obtain the following 
\begin{align*}
e_1\cdot e_3=be_4,\ \ e_2\cdot e_2=-ae_2,\ \  e_2\cdot e_4=ae_4,
\end{align*}  
and 
\begin{equation*}
e_3\cdot e_1=e_1\cdot e_4=e_4\cdot e_1=e_2\cdot e_3=e_3\cdot e_2=e_4\cdot e_2=e_3\cdot e_3=e_3\cdot e_4=e_4\cdot e_3=e_4\cdot e_4=0.
\end{equation*} 
Easily we can see that the Levi-Civita product computed in the above satisfies in (\ref{AM1}). Thus $(\mathfrak{g}, [\cdot,\cdot], \phi_{\mathfrak{g}}, K, <,>)$ is a para-K\"{a}hle hom-Lie algebra.
\end{example} 
\begin{theorem}\label{304}
Let $(\mathfrak{g},[\cdot, \cdot], \phi_{\mathfrak{g}}, <,>, K)$ be a para-K\"{a}hler hom-Lie algebra. Then 
 the following statements are verified
 
i)\  \ $(\mathfrak{g}, [\cdot, \cdot], { \phi_\mathfrak{g}}, \Omega)$ is a symplectic hom-Lie algebra, where 
\begin{equation}\label{*}
\Omega(u,v)=<({ \phi_\mathfrak{g}}\circ K)u,v>.
\end{equation}

ii)\  \ $\mathfrak{g}^1$ and $\mathfrak{g}^{-1}$ are subalgebras isotropic with respect to $< ,>$, and Lagrangian with
respect to $\Omega$,

iii) \ \ $(\mathfrak{g}, [\cdot, \cdot], { \phi‎‎_\mathfrak{g}}, <,>, K)$ is a para-Hermitian  hom-lie algebra,

iv)\  \ \ for any $u\in\mathfrak{g}$, $u.\mathfrak{g}^1\subset\mathfrak{g}^1 $ and $u.\mathfrak{g}^{-1}\subset\mathfrak{g}^{-1} $
(the dot is the Levi-Civita product),

v)\ \  for any $u\in\mathfrak{g}^1$, $\phi‎‎_\mathfrak{g}(u)\in\mathfrak{g}^1$ and for any $\bar{u}\in\mathfrak{g}^{-1}$, $\phi_\mathfrak{g}(\bar{u})\in\mathfrak{g}^{-1}$. 
\end{theorem}
\begin{proof}
Applying  (\ref{SL}) and (\ref{*}), we obtain
\begin{align}\label{AM300}
&\Omega([u,v],{ \phi‎‎_\mathfrak{g}}(w))+\Omega([v,w],{ \phi‎‎_\mathfrak{g}}(u))+\Omega([w,u],{ \phi‎‎_\mathfrak{g}}(v))\\
&=-<[u,v],({ \phi‎‎_\mathfrak{g}}\circ K) ({ \phi‎‎_\mathfrak{g}} (w))>-<[v,w],({ \phi‎‎_\mathfrak{g}}\circ K) ({ \phi‎‎_\mathfrak{g}} (u))>-<[w,u],({ \phi‎‎_\mathfrak{g}}\circ K) ({ \phi‎‎_\mathfrak{g}} (v))>\nonumber\\
&=-<u.v,({ \phi‎‎_\mathfrak{g}}\circ K) ({ \phi‎‎_\mathfrak{g}} (w))>+<v.u,({ \phi‎‎_\mathfrak{g}}\circ K) ({ \phi‎‎_\mathfrak{g}} (w))>-<v.w,({ \phi‎‎_\mathfrak{g}}\circ K) ({ \phi‎‎_\mathfrak{g}} (u))>\nonumber\\
&\ \ \ \ +<w.v,({ \phi‎‎_\mathfrak{g}}\circ K) ({ \phi‎‎_\mathfrak{g}} (u))>-<w.u,({ \phi‎‎_\mathfrak{g}}\circ K) ({ \phi‎‎_\mathfrak{g}} (v))>+<u.w,({ \phi‎‎_\mathfrak{g}}\circ K) ({ \phi‎‎_\mathfrak{g}}( v))>,\nonumber
\end{align}
for any $u,v,w\in\mathfrak{g}$. But using (\ref{L2}), (\ref{AM2}) and the skew-symmetric property of ${ \phi_\mathfrak{g}}\circ K$ with respect to $<,>$,  we acquire
\begin{align*}\label{AM301}
<u.v, ({ \phi‎‎_\mathfrak{g}}\circ K)({ \phi‎‎_\mathfrak{g}} (w))>=&<({ \phi‎‎_\mathfrak{g}}\circ K)(u.({ \phi‎‎_\mathfrak{g}}\circ K)(v)), ({ \phi‎‎_\mathfrak{g}}\circ K)({ \phi‎‎_\mathfrak{g}} (w))>\\
&=-<u.({ \phi‎‎_\mathfrak{g}}\circ K)v, ({ \phi‎‎_\mathfrak{g}} w)>=<u.w,{ \phi‎‎_\mathfrak{g}} ({ \phi‎‎_\mathfrak{g}}\circ K)(v)>.\nonumber
\end{align*}
Setting the above equation in (\ref{AM300}) we get
\begin{align*}
&\Omega([u,v],{ \phi‎‎_\mathfrak{g}}(w))+\Omega([v,w],{ \phi‎‎_\mathfrak{g}}(u))+\Omega([w,u],{ \phi‎‎_\mathfrak{g}}(v))\\
&=<u.({ \phi‎‎_\mathfrak{g}}\circ K)v, ({ \phi‎‎_\mathfrak{g}} w)>+<v.({ \phi‎‎_\mathfrak{g}}\circ K)u, ({ \phi‎‎_\mathfrak{g}} w)>-<v.w,({ \phi‎‎_\mathfrak{g}}\circ K) ({ \phi‎‎_\mathfrak{g}} u)>\\
&+<w.({ \phi‎‎_\mathfrak{g}}\circ K)v, ({ \phi‎‎_\mathfrak{g}} u)>-<w.u,({ \phi‎‎_\mathfrak{g}}\circ K) ({ \phi‎‎_\mathfrak{g}} v)>+<u.w,({ \phi‎‎_\mathfrak{g}}\circ K) ({ \phi‎‎_\mathfrak{g}} v)>\\
&=<u.w, ({ \phi‎‎_\mathfrak{g}}^2\circ K) v>+<v.w, ({ \phi‎‎_\mathfrak{g}}^2\circ K)u>-<v.w,K u>\\
&+<w.u, ({ \phi‎‎_\mathfrak{g}}^2\circ K)v>-<w.u,K v>+<u.w,K v>=0.
\end{align*}
Moreover, (\ref{*}) implies
\[
\Omega({ \phi‎‎_\mathfrak{g}}(u),{ \phi‎‎_\mathfrak{g}}(v))=<({ \phi‎‎_\mathfrak{g}}\circ K){ \phi‎‎_\mathfrak{g}}(u), { \phi‎‎_\mathfrak{g}}( v)>=<K(u),{ \phi‎‎_\mathfrak{g}}(v)>=<({ \phi‎‎_\mathfrak{g}}\circ K)u,v>=\Omega(u,v).
\]
Therefore we have $(i)$.

Since $({ \phi_\mathfrak{g}}\circ K)(u)=u$ and $({ \phi_\mathfrak{g}}\circ K)(v)=v$ for any $u,v\in\mathfrak{g}^1$, then $<({ \phi_\mathfrak{g}}\circ K)u,({ \phi_\mathfrak{g}}\circ K)v=<u,v>$. On the other hand, ${ \phi_\mathfrak{g}}\circ K$ is skew-symmetric with respect to $<,>$, hence we deduce $<u,v>=0$. Similarly, it follows that $<\bar{u},\bar{v}>=0$, for any $ \bar{u}, \bar{v}\in\mathfrak{g}^{-1}$. Thus $\mathfrak{g}^1$ and $\mathfrak{g}^{-1}$ are isotropic with respect to $<,>$. Now we prove the second part. Let $u\in \mathfrak{g}^1$. Since 
\[
\Omega(u,v)=<(\phi_{\mathfrak{g}^1}\circ K)(u),v>=<u,v>=0,
\]
for any $v\in\mathfrak{g}^1$, then
 $u\in (\mathfrak{g}^1)^\perp$ and so $\mathfrak{g}^1\subseteq (\mathfrak{g}^1)^\perp$. Now, we consider $0\neq u+\bar{u}\in (\mathfrak{g}^1)^\perp$ such that $u\in \mathfrak{g}^1$ and $\bar{u}\in \mathfrak{g}^{-1}$. Then, we have 
\begin{align*}
0=\Omega(u+\bar{u},v)=\Omega(u,v)+\Omega(\bar{u},v),\ \ \ \forall v\in \mathfrak{g}^1.
\end{align*}
Since $\Omega(u,v)=0$, it follows that $\Omega(\bar{u},v)=0$, for any $v\in\mathfrak{g}^1$. On the other hand, we have $\Omega(\bar{u},\bar{v})=0$, for any $\bar{v}\in \mathfrak{g}^{-1}$. Thus $\Omega(\bar{u},v+\bar{v})=0$, and so $\bar{u}=0$ (because $\Omega$ is non-degenerate). Therefore $(\mathfrak{g}^1)^\perp\subseteq \mathfrak{g}^1$ and consequently $(\mathfrak{g}^1)^\perp=\mathfrak{g}^1$. This shows that $\mathfrak{g}^1$ is a Lagrangian subspace of $\mathfrak{g}$ with respect to $\Omega$. In a similar way, we can deduce that $\mathfrak{g}^{-1}$ is a Lagrangian subspace of $\mathfrak{g}$ with respect to $\Omega$.

To prove (iii), at first we show that $K$ is a para-complex structure. Using (\ref{AM1}) and (\ref{AM2}), we get
\begin{align*}
4N(u,v)=&[({{ \phi‎‎_\mathfrak{g}}\circ K})u, ({{ \phi_\mathfrak{g}}\circ K})v]-({{ \phi_\mathfrak{g}}\circ K})[({{ \phi_\mathfrak{g}}\circ K})u, v]-({{ \phi_\mathfrak{g}}\circ K})[u,({{ \phi_\mathfrak{g}}\circ K})v]+[u,v]\\
=&\phi_\mathfrak{g}( Ku).\phi_\mathfrak{g}( Kv)-\phi_\mathfrak{g}( Kv).\phi_\mathfrak{g}( Ku)-\phi_\mathfrak{g}\circ K(\phi_\mathfrak{g}(Ku).v-v.\phi_\mathfrak{g}(Ku))\\
&-\phi_\mathfrak{g}\circ K(-\phi_\mathfrak{g}(Kv).u+u.\phi_\mathfrak{g}(Kv))+u.v-v.u=\phi_\mathfrak{g}( Ku).\phi_\mathfrak{g}( Kv)\\
&-\phi_\mathfrak{g}( Kv).\phi_\mathfrak{g}( Ku)-\phi_\mathfrak{g}( Ku).\phi_\mathfrak{g}( Kv)+v.u-u.v
+\phi_\mathfrak{g}( Kv).\phi_\mathfrak{g}( Ku)+u.v-v.u=0.
\end{align*}
On the other hand, using (ii) we have
\[
\dim\mathfrak{g}=\dim\mathfrak{g}^1+(\mathfrak{g}^1)^\perp=2\dim\mathfrak{g}^1,
\]
which gives $\dim\mathfrak{g}^1=\dim\mathfrak{g}^{-1}$. Thus $K$ is a para-complex structure on $(g, [\cdot, \cdot], \phi_\mathfrak{g})$. To complete the proof we check (\ref{EL}) for it. Since ${ \phi_\mathfrak{g}}\circ K$ is skew-symmetric with respect to $<,>$, then we have
\[
<({ \phi_\mathfrak{g}}\circ K)u, ({ \phi_\mathfrak{g}}\circ K)v>=-<u, ({ \phi_\mathfrak{g}}\circ K)^2v>=- <u , v>,
\]
which gives the assertion.

To prove (iv)
we consider $u\in\mathfrak{g}$ and $v\in \mathfrak{g}^1$ . Since $({ \phi‎‎_\mathfrak{g}}\circ K)v=v$, using (\ref{AM2}) we get
\[
({ \phi_\mathfrak{g}}\circ K)(u.v)=u.({ \phi_\mathfrak{g}}\circ K)v=u.v.
\]
The above equation implies $u.v\in \mathfrak{g}^1$. Similarly, it follows that $u\cdot\bar{u}\in\mathfrak{g}^{-1}$ for any $\bar{u}\in\mathfrak{g}^{-1}$.

To prove (v), we let $u\in\mathfrak{g}^1$. Since $(
K\circ{ \phi_\mathfrak{g}})(u)=u$ and $K\circ { \phi_\mathfrak{g}}={ \phi_\mathfrak{g}}\circ K$, it follows that
\[
(K\circ \phi_{\mathfrak{g}})(\phi_{\mathfrak{g}}(u))=(\phi_{\mathfrak{g}}\circ K\circ \phi_{\mathfrak{g}})(u)=\phi_{\mathfrak{g}}(u),
\]
which gives ${\phi_\mathfrak{g}}(u)\in \mathfrak{g}^1$. Similarly, we obtain ${ \phi_\mathfrak{g}}(\bar{u})\in \mathfrak{g}^{-1}$ for any $\bar{u}\in \mathfrak{g}^{-1}$.
\end{proof}
\begin{remark}
According to part (v) of the above theorem, we can write 
	$$\phi_{\mathfrak{g}}:{\mathfrak{g}^1\oplus\mathfrak{g}^{-1}}\rightarrow {\mathfrak{g}^1\oplus\mathfrak{g}^{-1}},$$  as $\phi_{\mathfrak{g}}(u+\bar{u})=\phi_{\mathfrak{g}^1}(u)+\phi_{\mathfrak{g}^{-1}}(\bar{u})$. Also, according to (i) and (iv) of the above theorem, a para-K\"{a}hler hom-Lie algebra has two products, the Levi-Civita product and the hom-left
	symmetric product $\bold{ a}$ associated with $(\mathfrak{g}, [\cdot, \cdot], \phi_{\mathfrak{g}}, \Omega)$.
	
\end{remark}
\begin{proposition}
If  $(\mathfrak{g}, [\cdot, \cdot], { \phi‎‎_\mathfrak{g}}, <,>, K, \Omega)$ is a para-K\"{a}hler hom-Lie algebra, then 
for any $u,v\in \mathfrak{g}^1$ and $\bar{u},\bar{v}\in\mathfrak{g}^{-1}$ we have 
\begin{align*}
u\cdot v=&\bold{a}(u,v), \ \ \ \ \ \ \ \ \ \ \ \  \ \ \ \ \ \ \ \  \bar{u}\cdot\bar{v}=\bold{a}(\bar{u},\bar{v}).
\end{align*}
\end{proposition}
\begin{proof}
We show that the
Levi-Civita product induces a product on $ \mathfrak{g}^1$ and $\mathfrak{g}^{-1}$, which coincides with the hom-left-symmetric
product $\bold{ a}$. Since $\mathfrak{g}$ is a symplectic hom-Lie algebra, then  
 using (\ref{AM302}) and  (\ref{AM10}), we get
\begin{align}\label{AM13}
0=\Omega(\bold{ a}(u,v),{ \phi‎‎_{\mathfrak{g}}}(w))+\Omega({ \phi‎‎_{\mathfrak{g}}}(v),[u,w])=\Omega(\bold{ a}(u,v),{ \phi‎‎_\mathfrak{g}}(w))+\Omega({ \phi‎‎_{\mathfrak{g}}}(v),u.w-w.u),
\end{align}
for any $w\in\mathfrak{g}$ and $u,v\in\mathfrak{g}^1$. On the other hand, (\ref{AM302}),  (\ref{*}) and parts (i) and (v) of Theorem \ref{304} imply
\begin{align}\label{AM12}
\Omega({ \phi‎‎_{\mathfrak{g}}}(v),u.w)=<({ \phi‎‎_{\mathfrak{g}}}\circ K)({ \phi‎‎_{\mathfrak{g}}(v)}),u.w>=<{ \phi‎‎_{\mathfrak{g}}(v)},u.w>=-<{ \phi‎‎_\mathfrak{g}}(w), u.v>=-\Omega(u.v,{ \phi‎‎_\mathfrak{g}}(w)),
\end{align}
and
\[
-\Omega({ \phi_{\mathfrak{g}}}(v),w.u)=<w.u,{ \phi_{\mathfrak{g}}}(v)>=<{ \phi_\mathfrak{g}}(w.u),v>=\Omega(v,{ \phi_\mathfrak{g}}(w.u))=\Omega({\phi_{\mathfrak{g}}}(v), w.u).
\]
The above equation gives
\begin{align}\label{AM11}
\Omega({ \phi_{\mathfrak{g}}}(v),w.u)=0.
\end{align}
Setting (\ref{AM12}) and (\ref{AM11}) in (\ref{AM13}) and using the non-degenerate property of $\Omega$ and $\phi_{\mathfrak{g}}$ we deduce $u\cdot v=\bold{a}(u,v)$. Similarly, we can obtain the second relation.
\end{proof}
\begin{corollary}
If $(\mathfrak{g}, [\cdot, \cdot], { \phi_\mathfrak{g}}, <,>, K, \Omega)$ is a para-K\"{a}hler hom-Lie algebra, then $(\mathfrak{g}^1, \bold{a},\phi_{\mathfrak{g}^1})$ and $(\mathfrak{g}^{-1}, \bold{a},\phi_{\mathfrak{g}^{-1}})$ are hom-left symmetric algebras.
\end{corollary}
Propositions \ref{3.9} and \ref{3.10} result the following
\begin{proposition}
Hom-left symmetric algebra structures on $(\mathfrak{g}^1, \bold{a},\phi_{\mathfrak{g}^1})$ and $(\mathfrak{g}^{-1}, \bold{a},\phi_{\mathfrak{g}^-{1}})$ induce hom-Lie algebra structures on $\mathfrak{g}^1$ and $\mathfrak{g}^{-1}$.
\end{proposition}
\section{Phase spaces and para-K\"{a}hler structures}
\begin{definition}
Let $(V, \cdot , \phi)$ be a hom-algebra and $V^*$ be its dual space. If there exists a hom-algebra structure on the direct sum of the underling vector space $(V, \phi)$ and $(V^*, \phi^*)$ ($\phi^*$ is the dual map of $\phi$, i.e., $\phi^*=\phi^t$) such that $(V,\phi)$ and $(V^*,\phi^*)$ are sub-hom-algebras and the natural skew-symmetry bilinear form $\omega$
on $V\oplus V^*$ given by 
\begin{align}\label{S1}
\omega(u+\alpha,v+\beta)=\prec\beta,u‎\succ-\prec\alpha,v‎\succ,
\end{align}
is a symplectic form for any $u,v\in V$ and $\alpha,\beta\in V^*$, then $(V\oplus V^*, \omega)$ is called phase space of $V$.
\end{definition}
\begin{definition}
Let $(\mathfrak{g}, [\cdot,\cdot],\phi_{\mathfrak{g}})$ be a hom-Lie algebra and $\mathfrak{g}^*$ be
its dual space. A phase space of $\mathfrak{g}$ is defined as a hom-Lie algebra $T^*\mathfrak{g}=(\mathfrak{g}\oplus‎\mathfrak{g}^*, [\cdot,\cdot]_{\mathfrak{g}\oplus‎\mathfrak{g}^*}, \phi_{\mathfrak{g}}\oplus\phi_{\mathfrak{g^*}})$ endowed with the symplectic form $\omega$ given by (\ref{S1}), where $\phi_{\mathfrak{g}^*}=(\phi_\mathfrak{g})^*$.
\end{definition}
Let $(\mathfrak{g}, [\cdot, \cdot], \phi_{\mathfrak{g}}, <,>, K,\Omega)$ be a para-K\"{a}hler hom-Lie algebra. For any $\bar{u}\in\mathfrak{g}^{-1}$, we consider $u^*\in(\mathfrak{g}^1)^*$ such that $\prec\bar{u}^*,v\succ=<\bar{u},v>$, for all $v\in\mathfrak{g}^1$. Therefore we can define map $\mathfrak{g}^{-1}\rightarrow(\mathfrak{g}^1)^*$, $\bar{u}\rightarrow\bar{u}^*$, which is an isomorphism.
\begin{lemma}
According to the identification $\mathfrak{g}^{-1}\simeq(\mathfrak{g^1})^*$, endomorphisms $\phi_{\mathfrak{g}^{-1}}$ and $(\phi_{\mathfrak{g}^{1}})^*$ are the same.
\end{lemma}
\begin{proof}
Let $\bar{u}\in\mathfrak{g}^{-1}$ and $\bar{u}^*$ be the corresponding element of it in $(\mathfrak{g^1})^*$. Then we have 
\[
\prec(\phi_{\mathfrak{g}^1})^*(\bar{u}^*), v\succ=\prec\bar{u}^*, \phi_{\mathfrak{g}^1}(v)\succ=<\bar{u}, \phi_{\mathfrak{g}^1}(v)>,\ \ \ \forall v\in\mathfrak{g}^1.
\] 
If we denote by $(\phi_{\mathfrak{g}^{-1}}(\bar{u}))^*$, the corresponding element of $\phi_{\mathfrak{g}^{-1}}(\bar{u})\in\mathfrak{g}^{-1}$ in $(\mathfrak{g^1})^*$, then using (\ref{EL1}) we get 
\[
\prec(\phi_{\mathfrak{g}^{-1}}(\bar{u}))^*, v\succ=<\phi_{\mathfrak{g}^{-1}}(\bar{u}), v>=<\bar{u}, \phi_{\mathfrak{g}^1}(v)>, \ \ \ \forall v\in\mathfrak{g}^1.
\]
Two above equations show that $(\phi_{\mathfrak{g}^1})^*(\bar{u}^*)=(\phi_{\mathfrak{g}^{-1}}(\bar{u}))^*$. 
\end{proof}
\begin{remark}
In the next, according to the identification $\mathfrak{g}^{-1}\simeq(\mathfrak{g^1})^*$, we denote the elements of $\mathfrak{g}^{-1}$ or $(\mathfrak{g^1})^*$ by $\alpha, \beta, \gamma$. Also, to simplify we use ${ \phi_{(\mathfrak{g}^1)^*}}$ instead  $(\phi_{\mathfrak{g}^{1}})^*$.
\end{remark}
\begin{proposition}
Let $(\mathfrak{g}, [\cdot, \cdot], \phi_{\mathfrak{g}}, <,>, K,\Omega)$ be a para-K\"{a}hler hom-Lie algebra and for
  any $u\in\mathfrak{g}^1, \alpha\in(\mathfrak{g}^{1})^*$,  $L_u$ and $L_\alpha$ be the left
multiplication operators on $u\in\mathfrak{g}^1$ and $\alpha\in(\mathfrak{g}^{1})^*$, respectively (i.e., $L_uv=u.v$ and $L_\alpha\beta=\alpha.\beta$ for any $v\in\mathfrak{g}^1,\beta\in(\mathfrak{g}^{1})^*$). Then 
$(\mathfrak{g}^1,\phi_{\mathfrak{g}^1},L)$ and $((\mathfrak{g}^1)^*,\phi_{(\mathfrak{g}^1)^*},L)$ give representations of the hom-Lie algebra $\mathfrak{g}^1$ and $(\mathfrak{g}^{1})^*$, respectively, where
$L:\mathfrak{g}^1\rightarrow gl(\mathfrak{g}^1)$ with $u\rightarrow L_u$ and $L:(\mathfrak{g}^1)^*\rightarrow gl((\mathfrak{g}^1)^*)$ with $\alpha\rightarrow L_\alpha$, .  
\end{proposition}
\begin{proof}
We must to show that 
\begin{equation}\label{5.2}
\left\{
\begin{array}{cc}
\ \ \ \ \ \ \ \  \ \ \ \ \ \ \ \ &\hspace{-5cm}L_{{ \phi‎‎_{\mathfrak{g}^1}}(u)}\circ{ \phi‎‎_{\mathfrak{g}^1}}={ \phi‎‎_{\mathfrak{g}^1}}\circ L_u,\\
\ \ \ \ \ \ \ \  \ \ &\hspace{-2.3cm}L_{[u,v]}\circ { \phi‎‎_{\mathfrak{g}^1}}=L_{{ \phi‎‎_{\mathfrak{g}^1}}(u)}\circ L_v-L_{ { \phi‎‎_{\mathfrak{g}^1}}‎(v)}\circ L_u,
\end{array}
\right.
\end{equation}
and 
\begin{equation}\label{5.3}
\left\{
\begin{array}{cc}
&\hspace{-2.6cm}L_{{ \phi_{(\mathfrak{g}^1)^*}}(\alpha)}\circ{ \phi_{(\mathfrak{g}^1)^*}}={ \phi_{(\mathfrak{g}^1)^*}}\circ L_\alpha,\\
&\hspace{-.2cm}L_{[\alpha,\beta]}\circ { \phi_{(\mathfrak{g}^1)^*}}
=L_{{ \phi_{(\mathfrak{g}^1)^*}}(\alpha)}\circ L_\beta-L_{ { \phi_{(\mathfrak{g}^1)^*}}‎(\beta)}\circ L_\alpha.
\end{array}
\right.
\end{equation}
If we consider the hom-left-symmetric algebra $\mathfrak{g}^1$, then for any $u,v,w\in\mathfrak{g}^1$ we have
\begin{align*}
\bold{a}(\bold{a}(u,v),{ \phi‎‎_{\mathfrak{g}^1}}(w))-\bold{a}(\bold{a}({ \phi‎‎_{\mathfrak{g}^1}}(u),\bold{a}(v,w))=\bold{a}(\bold{a}(v,u),{ \phi_{\mathfrak{g}^1}}(w))-\bold{a}({ \phi_{\mathfrak{g}^1}}(v),\bold{a}(u,w)),
\end{align*}
where $u.v=\bold{a}(u,v)$ and $\bold{a}(u,v)-\bold{a}(v,u)=[u,v]$. Since  $L_uv=u\cdot v$, the above equation can be written as follows
 \[
L_{[u,v]}\circ { \phi_{\mathfrak{g}^1}}=L_{{ \phi_{\mathfrak{g}^1}}(u)}\circ L_v-L_{ { \phi_{\mathfrak{g}^1}}‎(v)}\circ L_u.
\]
Moreover, we have ${ \phi_{\mathfrak{g}^1}}(u.v)={ \phi_{\mathfrak{g}^1}}(u)\cdot { \phi_{\mathfrak{g}^1}}(v)$, i.e., 
$L_{{ \phi_{\mathfrak{g}^1}}(u)}{ \phi_{\mathfrak{g}^1}}(v)={ \phi_{\mathfrak{g}^1}}L_uv$. Thus we have (\ref{5.2}). Similarly, we can deduce (\ref{5.3}). 
\end{proof}
\begin{corollary}
Let $(\mathfrak{g}, [\cdot, \cdot], \phi_{\mathfrak{g}}, <,>, K,\Omega)$ be a para-K\"{a}hler hom-Lie algebra. Then representation $(\mathfrak{g}^1,\phi_{\mathfrak{g}^1},L)$ is admissible i.e.,  $((\mathfrak{g}^1)^*,\phi^*_{\mathfrak{g}^1},\widetilde{L})$ is a  representation of $\mathfrak{g}$.
\end{corollary}
\begin{proof}
According to (\ref{AM200}), it is sufficient to prove the following
\begin{equation*}
\left\{
\begin{array}{cc}
 &\hspace{-3cm}{ \phi_{\mathfrak{g}^1}}\circ L_{{ \phi_{\mathfrak{g}^1}}(u)}=L_u\circ { \phi_{\mathfrak{g}^1}},\\
&{ \phi_{\mathfrak{g}^1}}\circ (L_{[u,v]})=L_u\circ L_{ \phi_{\mathfrak{g}^1}(v)}-L_v\circ L_{ \phi_{\mathfrak{g}^1}(u)}.
\end{array}
\right.
\end{equation*}
If $u,v\in \mathfrak{g}^1$, then using ${ \phi_{\mathfrak{g}^1}}(u\cdot v)={ \phi_{\mathfrak{g}^1}}(u)\cdot { \phi_{\mathfrak{g}^1}}(v)$, we infer ${ \phi_{\mathfrak{g}^1}}({ \phi_{\mathfrak{g}^1}}(u)\cdot v)=u\cdot { \phi_{\mathfrak{g}^1}}(v)$. Setting ${ \phi_{\mathfrak{g}^1}}(u)\cdot v=L_{ \phi_{\mathfrak{g}^1}(u)}v$ in it, we deduce ${ \phi_{\mathfrak{g}^1}}( L_{{ \phi_{\mathfrak{g}^1}}(u)}v)=L_u { \phi_{\mathfrak{g}^1}}(v)$.
On the other hand, for any $w\in\mathfrak{g}^1$ we have 
\[
(u\cdot v)\cdot{ \phi_{\mathfrak{g}^1}}(w)-(v\cdot u)\cdot{ \phi_{\mathfrak{g}^1}}(w)={ \phi_{\mathfrak{g}^1}}(u)\cdot(v\cdot w)-{ \phi_{\mathfrak{g}^1}}(v)\cdot(u\cdot w).
\]
Contracting the above equation with ${ \phi_{\mathfrak{g}^1}}$ implies 
\[
{ \phi_{\mathfrak{g}^1}}((u\cdot v)\cdot{ \phi_{\mathfrak{g}^1}}(w)-(v\cdot u)\cdot{ \phi_{\mathfrak{g}^1}}(w))=u\cdot{ \phi_{\mathfrak{g}^1}}(v\cdot w)-v\cdot{ \phi_{\mathfrak{g}^1}}(u\cdot w).
\]
Setting $L_uv=u\cdot v$ in the above equation follows that
\[
{ \phi_{\mathfrak{g}^1}}(L_{[u,v]}{ \phi_{\mathfrak{g}^1}}(w))=L_uL_{ \phi_{\mathfrak{g}^1}(v)}{ \phi_{\mathfrak{g}^1}}(w)-L_vL_{ \phi_{\mathfrak{g}^1}(u)}{ \phi_{\mathfrak{g}^1}}(w).
\]
\end{proof}
\begin{lemma}
Let $(\mathfrak{g}, [\cdot, \cdot], \phi_{\mathfrak{g}}, <,>, K,\Omega)$ be a para-K\"{a}hler hom-Lie algebra.
Then  $T^*\mathfrak{g}^1=\mathfrak{g}^1\oplus(\mathfrak{g}^1)^*$ is a hom-Lie algebra.
\end{lemma}
\begin{proof}
For any $u,v\in\mathfrak{g}^1$ and $\alpha,\beta\in(\mathfrak{g}^1)^*$ we consider a bracket on the vector space $\mathfrak{g}^1\oplus(\mathfrak{g}^1)^*$ as follows 
\[
[u+\alpha,v+\beta]:=[u,v]_{\mathfrak{g}^1}+\widetilde{L}_u\beta-\widetilde{L}_v\alpha.
\]
Also, we define morphism $\phi_{\mathfrak{g}}: \mathfrak{g}^1\oplus(\mathfrak{g}^1)^*\rightarrow {\mathfrak{g}^1}\oplus(\mathfrak{g}^1)^*$ by
\[
\phi_{\mathfrak{g}}(u+\alpha)=\phi_{\mathfrak{g}^1}(u)+{\phi_{(\mathfrak{g}^1)^*}}(\alpha).
\]
Obviously, we have 
\[
[u+\alpha,v+\beta]=-[v+\beta,u+\alpha].
\]
Now, we check that the hom-Jacobi identity for it. We have
\begin{align*}
&\circlearrowleft_{(u,\alpha),(v,\beta),(w,\gamma)}[\phi_{\mathfrak{g}}(u+\alpha),[v+\beta,w+\gamma]]=\circlearrowleft_{(u,\alpha),(v,\beta),(w,\gamma)}[\phi_{\mathfrak{g}^1}(u)+{\phi_{(\mathfrak{g}^1)^*}}(\alpha),[v,w]_{\mathfrak{g}^1}+\widetilde{L}_v\gamma-\widetilde{L}_w\beta]\\
&=\circlearrowleft_{(u,\alpha),(v,\beta),(w,\gamma)}\big([\phi_{\mathfrak{g}^1}(u),[v,w]_{\mathfrak{g}^1}]_{\mathfrak{g}^1}+\widetilde{L}_{\phi_{\mathfrak{g}^1}(u)}(\widetilde{L}_v\gamma-\widetilde{L}_w\beta)-\widetilde{L}_{[v,w]_{\mathfrak{g}^1}}{\phi_{(\mathfrak{g}^1)^*}}(\alpha)\big).
\end{align*}
On the other hand, $((\mathfrak{g}^1)^* ,{\phi_{(\mathfrak{g}^1)^*}}, \widetilde{L})$ is a representation of $\mathfrak{g}^1$. Therefore we have 
\[
\widetilde{L}_{[v,w]_{\mathfrak{g}^1}}\circ {\phi_{(\mathfrak{g}^1)^*}}=\widetilde{L}_{\phi_{\mathfrak{g}^1}(v)}\circ \widetilde{L}_w-\widetilde{L}_{\phi_{\mathfrak{g}^1}(w)}\circ \widetilde{L}_v,
\]
Two above equations imply 
\begin{align*}
&\circlearrowleft_{(u,\alpha),(v,\beta),(w,\gamma)}[\phi_{\mathfrak{g}}(u+\alpha),[v+\beta,w+\gamma]]
=\circlearrowleft_{(u,\alpha),(v,\beta),(w,\gamma)}[\phi_{\mathfrak{g}^1}(u),[v,w]]\\
&+\widetilde{L}_{\phi_{\mathfrak{g}^1}(u)}(\widetilde{L}_v\gamma-\widetilde{L}_w\beta)+\widetilde{L}_{\phi_{\mathfrak{g}^1}(v)}(\widetilde{L}_w\alpha-\widetilde{L}_u\gamma)+\widetilde{L}_{\phi_{\mathfrak{g}^1}(w)}(\widetilde{L}_u\beta-\widetilde{L}_v\alpha)-\widetilde{L}_{\phi_{\mathfrak{g}^1}(v)}\widetilde{L}_w\alpha\\
&+\widetilde{L}_{\phi_{\mathfrak{g}^1}(w)} \widetilde{L}_v\alpha-\widetilde{L}_{\phi_{\mathfrak{g}^1}(w)} \widetilde{L}_u\beta)+\widetilde{L}_{\phi_{\mathfrak{g}^1}(u)}\widetilde{L}_w\beta
-\widetilde{L}_{\phi_{\mathfrak{g}^1}(u)}\widetilde{L}_v\gamma+\widetilde{L}_{\phi_{\mathfrak{g}^1}(v)}\widetilde{L}_u\gamma=0.
\end{align*}
\end{proof}
\begin{lemma}
Let $(\mathfrak{g}, [\cdot, \cdot], \phi_{\mathfrak{g}}, <,>, K,\Omega)$ be a para-K\"{a}hler hom-Lie algebra. Then hom-Lie algebra $T^*\mathfrak{g}^1=\mathfrak{g}^1\oplus(\mathfrak{g}^1)^*$ is a phase space of $\mathfrak{g}^1$.
\end{lemma}
\begin{proof}
Let $u,v\in\mathfrak{g}^1$ and $\alpha,\beta\in(\mathfrak{g}^1)^*$. Then using the isotropic property of $\mathfrak{g}^1$ and $(\mathfrak{g}^1)^*$ we have
\begin{align*}
\Omega(u+\alpha,v+\beta)=<(\phi_{\mathfrak{g}}\circ K)(u+\alpha),v+\beta>=<u-\alpha,v+\beta>=-\prec\alpha,v‎\succ+\prec\beta,u‎\succ.
\end{align*}
On the other hand, we obtain
\begin{align*}
&\Omega(\phi_{\mathfrak{g}}(u+\alpha),\phi_{\mathfrak{g}}(v+\beta))=<(\phi_{\mathfrak{g}}\circ K)(\phi_{\mathfrak{g}^1}(u)+\phi_{(\mathfrak{g}^1)^*}\alpha),\phi_{\mathfrak{g}}(v+\beta)>\\
&=<\phi_{\mathfrak{g}^1}(u)-\phi_{(\mathfrak{g}^1)^*}\alpha,\phi_{\mathfrak{g}^1}(v)+\phi_{(\mathfrak{g}^1)^*}(\beta)>=<\phi_{\mathfrak{g}^1}(u),\phi_{(\mathfrak{g}^1)^*}(\beta)>-<\phi_{(\mathfrak{g}^1)^*}\alpha,\phi_{\mathfrak{g}^1}(v)>\\
&=<u,\beta>-<\alpha,v>=-\prec\alpha,v‎\succ+\prec\beta,u‎\succ=\Omega(u+\alpha,v+\beta).
\end{align*}
Since $\mathfrak{g}$ is a symplectic hom-Lie algebra, it follows that 
\[
\Omega([u+\alpha,v+\beta],\phi_\mathfrak{g}(w+\gamma))+\Omega([v+\beta,w+\gamma],\phi_\mathfrak{g}(u+\alpha))+\Omega([w+\gamma,u+\alpha],\phi_\mathfrak{g}(v+\beta))=0.
\] 
\end{proof}
\begin{proposition}
Let $(\mathfrak{g}, [\cdot, \cdot], \phi_{\mathfrak{g}}, <,>, K,\Omega)$ be a para-K\"{a}hler hom-Lie algebra. For any $u,v\in \mathfrak{g}^1$ and $\alpha\in(\mathfrak{g}^{1})^*$, we have 
\begin{align*}
u\cdot \alpha=&-L^t_{ \phi_{\mathfrak{g}^1}(u)}\alpha,\ \ \ \ \ \ \ \ \ \ \ \ \ \ \ \ \ \ \ \ \ \ \ \alpha\cdot u=-L^t_{ \phi_{(\mathfrak{g}^1)^*}(\alpha)} u,\\
\bold{a}(u, \alpha)=&R_{{ \phi_{(\mathfrak{g}^1)^*}(\alpha)}}^tu-ad^t_{ \phi_{\mathfrak{g}^1}(u)}\alpha, \ \ \ \ \ \ \ \ \ \bold{a}(\alpha,u)=-ad^t_{ \phi_{(\mathfrak{g}^1)^*}(\alpha)} u+R^t_{ \phi_{\mathfrak{g}^1}(u)}\alpha.
\end{align*}
\end{proposition}
\begin{proof}
Let $u,v\in \mathfrak{g}^1$ and $\alpha\in{(\mathfrak{g}^{1})^*}$. Then using (\ref{AM10}), Corollary \ref{AM350} and Koszul's formula we deduce  
\begin{align*}
&2<u\cdot \alpha,\phi‎‎_\mathfrak{g}(v)>=<[u,\alpha],\phi‎‎_\mathfrak{g}(v)>+<[v,u],\phi‎‎_{(\mathfrak{g}^{1})^*}(\alpha)>+<[v,\alpha],\phi‎‎_\mathfrak{g}(u)>\\
&=-\Omega(a(u,v),\phi‎‎_{(\mathfrak{g}^{1})^*}(\alpha))-\Omega(\alpha,\phi‎‎_\mathfrak{g}[v,u])-\Omega(a(v,u),\phi‎‎_{(\mathfrak{g}^{1})^*}(\alpha))
=-2<u\cdot v,\phi‎‎_{(\mathfrak{g}^{1})^*}(\alpha)>\\
&=-2<\phi‎‎_\mathfrak{g}(u\cdot v),\alpha>=-2<\phi‎‎_\mathfrak{g}(u)\cdot \phi‎‎_\mathfrak{g}(v),\alpha>=-2<L^t_{\phi‎‎_\mathfrak{g}(u)}\alpha,\phi‎‎_\mathfrak{g}(v)>.
\end{align*}
As $<,>$ and $\phi_{\mathfrak{g}^1}$ are non-degenerate, then $u\cdot\alpha=-L^t_{ \phi_{\mathfrak{g}^1}(u)}\alpha$. Similarly we get $\alpha\cdot u=-L^t_{ \phi_{(\mathfrak{g}^1)^*}(\alpha)} u$.\\
Considering $w\in \mathfrak{g}^1$ and $\gamma\in(\mathfrak{g}^1)^*$, we can write
\begin{align}\label{SS3}
\Omega(\bold{a}(u,\alpha),\phi_\mathfrak{g}(w+\gamma))=\Omega(\bold{a}(u,\alpha),\phi_\mathfrak{g}(w))+\Omega(\bold{a}(u,\alpha),\phi_{(\mathfrak{g}^{1})^*}(\gamma)).
\end{align}
But using (\ref{AM10}) and (\ref{AM302}), we get
\begin{align*}
&\Omega(\bold{a}(u,\alpha),\phi_\mathfrak{g}(w))=-\Omega(\phi_{(\mathfrak{g}^{1})^*}(\alpha),[u,w])=<\phi_{(\mathfrak{g}^{1})^*}(\alpha),[u,w]>=<\phi_{(\mathfrak{g}^{1})^*}(\alpha),ad_uw>=<ad_u^t\phi_{(\mathfrak{g}^{1})^*}(\alpha),w>\\
&=<\phi_\mathfrak{g}(ad^t_u\phi_{(\mathfrak{g}^{1})^*}(\alpha)),\phi_\mathfrak{g}(w)>=<ad^t_{\phi_\mathfrak{g}(u)}\alpha,\phi_\mathfrak{g}(w)>=-\Omega(ad^t_{\phi_\mathfrak{g}(u)}\alpha,\phi_\mathfrak{g}(w)),
\end{align*}
and
\begin{align*}
&\Omega(\bold{a}(u,\alpha),\phi_{(\mathfrak{g}^{1})^*}(\gamma))=-\Omega(\phi_{(\mathfrak{g}^{1})^*}(\alpha),[u,\gamma])=\Omega(\phi_{(\mathfrak{g}^{1})^*}(\alpha),[\gamma,u])=-\Omega(\bold{a}(\gamma,\alpha),\phi_\mathfrak{g}(u))\\
&=\Omega(\phi_\mathfrak{g}(u),\gamma\cdot \alpha)=<u,\phi_{(\mathfrak{g}^{1})^*}(\gamma)\cdot \phi_{(\mathfrak{g}^{1})^*}(\alpha)>=<u,R_{\phi_{(\mathfrak{g}^{1})^*} (\alpha)}\phi_{(\mathfrak{g}^{1})^*}(\gamma)>=<\phi_{(\mathfrak{g}^{1})^*}(\gamma),R^t_{\phi_{(\mathfrak{g}^{1})^*}(\alpha)}u>\\
&=\Omega(R^t_{\phi_{(\mathfrak{g}^{1})^*}(\alpha)}u,\phi_{(\mathfrak{g}^{1})^*}(\gamma)).
\end{align*}
Setting two above equations in (\ref{SS3}), we get 
$\bold{a}(u,\alpha)=R_{{ \phi‎‎_{(\mathfrak{g}^1)^*}(\alpha)}}^tu-ad^t_{ \phi‎‎_{\mathfrak{g}^1}(u)}\alpha$. Similarly, we get the last equation of the assertion.
\end{proof}
\begin{corollary}
If $u\in\mathfrak{g}^1$ and $\alpha\in{(\mathfrak{g}^1)^*}$, then 
\begin{align*}
 \phi‎‎_\mathfrak{g}(u\cdot \alpha)=&\phi‎‎_{\mathfrak{g}^1}(u)\cdot\phi‎‎_{(\mathfrak{g}^1)^*}(\alpha),\ \ \ \ \ \ \ \ \ \ \ \ \ \ \  \phi‎‎_{\mathfrak{g}}(\alpha\cdot u)=\phi‎‎_{(\mathfrak{g}^1)^*}(\alpha)\cdot\phi‎‎_{\mathfrak{g}^1}(u).
\end{align*}
\end{corollary}
\begin{proof}
Since $L$ is a representation, using the above proposition we get
\[
\phi‎‎_\mathfrak{g}(u\cdot \alpha)=-\phi‎‎_\mathfrak{g}(\L^t_{ \phi‎‎_{\mathfrak{g}^1}(u)}\alpha)=-\L^t_{ \phi‎‎^2_{\mathfrak{g}^1}(u)}\phi_{(\mathfrak{g}^1)^*}(\alpha)=\phi‎‎_{\mathfrak{g}^1}(u)\cdot\phi‎‎_{(\mathfrak{g}^1)^*}(\alpha).
\]
In a similar way, the second part of the assertion yields.
\end{proof}
Conversely, let $V$ be a finite dimensional vector space and $V^*$ be its dual space. We
suppose that $(V,\cdot, {\phi})$ and $(V^*,\cdot,{\phi^*})$ are involutive hom-left-symmetric algebras where $\phi^*=\phi^t$. We extend
the products of $V$ and $V^*$ to $V\oplus V^*$, by putting
\begin{align}\label{N1}
(u+\alpha)\cdot(v+\beta)=u\cdot v-L^t_{{{\phi^*}}(\alpha)}v-L^t_{  { \phi}(u)}\beta+\alpha\cdot\beta,
\end{align}
for any $u,v\in V$ and for any $\alpha,\beta\in V^*$. Also, we define linear map $\Phi:V\oplus V^*\rightarrow V\oplus V^*$ by 
\begin{align}\label{SS16}
\Phi(u+\alpha)= { \phi}(u)+{\phi^*}(\alpha), \ \ \forall u\in V,\ \alpha\in V^*,
\end{align}
which gives  $$\Phi^2(u+\alpha)= \Phi({ \phi}(u)+{\phi^*}(\alpha))={ \phi}^2(u)+{(\phi^*)^2}(\alpha)=u+\alpha.$$
\begin{corollary}
Let $u,v\in V$ and $\alpha,\beta\in V^*$. Then 
\begin{align}\label{AM600}
‎‎\prec ‎u\cdot\alpha,\phi(v)\succ =-‎‎\prec u\cdot v,\phi^*(\alpha)‎\succ‎‎‎‎, \ \ \ \ \ \ ‎‎\prec‎\alpha\cdot u,\phi^*(\beta)\succ=-‎‎\prec \alpha\cdot\beta,\phi(u)‎\succ‎‎‎‎.
\end{align}
\end{corollary}
\begin{proof}
Using (\ref{N1}), we obtain
\begin{align*}
\phi^*(\alpha)(u\cdot v)=&‎‎\prec\phi^*(\alpha), u\cdot v‎\succ‎‎=‎‎\prec \phi^t(\alpha),u\cdot v‎\succ=‎‎\prec \alpha,\phi(u\cdot v)‎\succ=‎‎\prec \alpha,\phi(u)\cdot\phi( v)‎\succ\\
‎&=\prec\alpha,L_{\phi(u)}\phi( v)‎\succ‎=‎‎\prec L_{\phi(u)}^t\alpha,\phi(v) ‎\succ‎‎=-‎‎\prec‎ u\cdot\alpha,\phi(v)\succ‎‎,
\end{align*}
for any $u,v\in V$ and $\alpha\in V^*$. Similarly, we get the second part of the assertion.
\end{proof}
\begin{theorem}\label{AM500}
Let $(V, \cdot, { \phi})$ and $(V^*,\cdot,{ {\phi^*}})$ be involutive hom-left-symmetric algebras. Then
 $(V\oplus V^*,\cdot,\Phi)$ is a hom-algebra, where $\cdot$ and $\Phi$ are given by (\ref{N1}) and (\ref{SS16}), respectively.
\end{theorem}
\begin{proof}
Using (\ref{SS16}), we have
 \begin{align*}
\Phi(u+\alpha)\cdot\Phi(v+\beta)=&\big({ \phi}(u)+{\phi^*}(\alpha)\big)\cdot\big({ \phi}(v)+{\phi^*}(\beta)\big)\\
=&{ { \phi}}(u)\cdot{ { \phi}}(v)+{ {\phi^*}}(\alpha)\cdot{ { \phi}}(v)+{ { \phi}}(u)\cdot{ {\phi^*}}(\beta)+{ {\phi^*}}(\alpha)\cdot{ {\phi^*}}(\beta),
\end{align*}
 for any $u+\alpha, v+\beta\in V\oplus V^*$, where $u,v\in V$ and $\alpha, \beta\in V^*$.
Since $V$ and $V^*$ are  hom-left-symmetric algebras, it implies
$
{ \phi}(u\cdot v)={ \phi}(u)\cdot{ \phi}(v)$ and ${\phi^*}(\alpha\cdot\beta)={\phi^*}(\alpha)\cdot{\phi^*}(\beta).
$
On the other hand, ${ { \phi}}(u)\in V$ and ${ {\phi^*}}(\alpha)\in V^*$. Thus 
$
{ { \phi}}(u)\cdot{ {\phi^*}}(\beta)=-L^t_{u}{ {\phi^*}}(\beta)$. Similarly, we get  
${ {\phi^*}}(\alpha)\cdot{ { \phi}}(v)=-L^t_{\alpha}{ { \phi}}(v)$. 
The above equations and explanations deduce 
 \begin{align*}
\Phi(u+\alpha)\cdot\Phi(v+\beta)
=&{ { \phi}}(u\cdot v)-L^t_{\alpha}{ { \phi}}(v)-L^t_{u}{ {\phi^*}}(\beta)+{ {\phi^*}}(\alpha\cdot\beta).
\end{align*}
Moreover $L$ is an admissible representation on $V$. So from the above equation, we obtain
 \begin{align*}
\Phi(u+\alpha)\cdot\Phi(v+\beta)
=&{ { \phi}}(u\cdot v)-\Phi(L^t_{{ {\phi^*}}(\alpha)}v)-\Phi (L^t_{{ { \phi}}(u)}\beta)+{ {\phi^*}}(\alpha\cdot\beta)\\
&=\Phi\big(u\cdot v-L^t_{{{\phi^*}}(\alpha)}v-L^t_{ { \phi}(u)}\beta+\alpha\cdot\beta\big)=\Phi(u+\alpha)\cdot\Phi(v+\beta).
\end{align*}
\end{proof}
Hom-algebra $(V\oplus V^*,\cdot,\Phi)$ satisfying Theorem \ref{AM500} is called  \textit{extendible hom-algebra}.
\begin{proposition}
If we consider two bilinear maps ${\rho} : V\times V^*\rightarrow \End(V)$ and ${\rho}^* : V^*\times V\rightarrow \End(V^*)$ 
defined by
\begin{align*}
{\rho}(u,\alpha)=&-L_{ { \phi}(u)}L_{ {\phi^*}‎‎(\alpha)}^t+L^t_{ \alpha} L_u-{ \Phi}\circ L^t_{L^t_{ { \phi}(u)}\alpha}-L_{L^t_{ {\phi^*}(\alpha)}u}\circ{ \phi},\\
{{\rho}^*}(\alpha,u)=&-L_{ {\phi^*}‎‎(\alpha)}L_{ { \phi}(u)}^t+L_{ u}^tL_\alpha-\Phi\circ L^t_{L^t_{ {\phi^*}(\alpha)}u}-L_{L^t_{ { \phi}(u)}\alpha}\circ{ {\phi^*}},
\end{align*}
then the endomorphism ${\rho}^*(\alpha,u)$ is the dual of ${\rho}(u,\alpha)$.
\end{proposition}
\begin{proof}
Using the definition of ${\rho}(u,\alpha)$ and (\ref{N1}), we obtain
\begin{align*}
{\rho}(u,\alpha)v=&-L_{ { { \phi}}‎(u)}L_{ { {\phi^*}}(\alpha)}^tv+L^t_{ \alpha} L_uv-{ \Phi‎‎}( L^t_{L^t_{ { { \phi}}‎(u)}\alpha}v)-L_{L^t_{ { {\phi^*}}(\alpha)}u}{ \phi‎‎}(v)\\
=&L_{ { { \phi}}‎(u)}(\alpha\cdot v)+L^t_{ ({ {\phi^*}})^2(\alpha)} (u\cdot v)+{{ \Phi}}( L^t_{u\cdot \alpha}v)+L_{\alpha\cdot u}{ { { \phi}}}(v)\\
=&{ { { \phi}}‎‎(u)}\cdot (\alpha\cdot v)-{ { {\phi^*}}(\alpha)}\cdot ( u\cdot v)-({u\cdot \alpha})\cdot { { { \phi}}‎‎(v)}+({\alpha\cdot u})\cdot { { { \phi}}‎}(v),\ \ \forall v\in V.
\end{align*}
The above equation implies
\begin{align*}
&‎‎\prec {\rho}(u,\alpha)v,\beta\succ =‎‎\prec { { { \phi}}‎(u)}\cdot (\alpha\cdot v)-{ { {\phi^*}}(\alpha)}\cdot ( u\cdot v)-({u\cdot \alpha})\cdot { { { \phi}}‎(v)}+({\alpha\cdot u}){ { { \phi}}‎}(v),\beta\succ \\
&=‎‎\prec \Phi({ { { \phi}}‎(u)}\cdot (\alpha\cdot v)-{ { {\phi^*}}(\alpha)}\cdot ( u\cdot v)-({u\cdot \alpha})\cdot { { { \phi}}‎‎(v)}+({\alpha\cdot u}){ { { \phi}}‎}(v)),\Phi(\beta)\succ \\
&=‎‎\prec u\cdot ({ {\phi^*}}(\alpha)\cdot { { \phi}}‎(v))-\alpha\cdot (( { { \phi}}‎(u)\cdot { { \phi}}(v))-{ {\phi^*}}({u\cdot \alpha})\cdot v+{ { \phi}}({\alpha\cdot u})\cdot v,{ {\phi^*}}(\beta)\succ \\
&=-‎‎\prec { {\phi^*}}(u\cdot \beta),{ {\phi^*}}(\alpha)\cdot { { \phi}}‎(v)\succ +‎‎\prec { {\phi^*}}(\alpha\cdot \beta),( { { \phi}}(u)\cdot { { \phi}}‎(v))\succ +‎‎\prec { {\phi^*}}({u\cdot \alpha})\cdot \beta,{ { \phi}}‎(v)\succ \\
&-‎‎\prec { { \phi}}‎({\alpha\cdot u})\cdot \beta,{ { \phi}}‎(v)\succ 
=‎‎\prec { {\phi^*}}(\alpha)\cdot(u\cdot \beta)-{ { \phi}}(u)\cdot(\alpha\cdot \beta)+({u\cdot \alpha})\cdot { {\phi^*}}(\beta)-({\alpha\cdot u})\cdot { {\phi^*}}(\beta),v\succ \\
&=‎‎\prec -L_{ { {\phi^*}}(\alpha)}L_{ { { \phi}}‎(u)}^t\beta+L_u^tL_\alpha\beta-\Phi(L^t_{L^t_{ { {\phi^*}}(\alpha)}u}\beta)-L_{L^t_{ { { \phi}}(u)}\alpha}{ { {\phi^*}}}(\beta),v\succ =‎‎\prec {{\rho}^*}(\alpha,u)\beta,v\succ ,
\end{align*}
for any $\beta\in V^*$.
\end{proof}
\begin{theorem}\label{AM501}
On an extendible hom-algebra $(V\oplus V^*,\cdot,\Phi)$, tensor curvature $\mathcal{K}$ of the product (\ref{N1}), 
satisfies the following
\begin{align*}
\mathcal{K}(u,v)w=\mathcal{K}(u,v)\alpha=\mathcal{K}(\alpha,\beta)w=\mathcal{K}(\alpha,\beta)\gamma=0,
\end{align*}
and 
\begin{align*}
\mathcal{K}(u,\alpha)v={\rho}(u,\alpha)v,\ \ \ \ \ \ \ \ \  \mathcal{K}(\alpha, u)\beta={\rho}^*(\alpha, u)\beta,
\end{align*}
for any $u,v,w\in V$ and $\alpha,\beta,\gamma\in V^*$.
\end{theorem}
\begin{proof}
We consider $u,v,w\in V$ and $\alpha,\beta,\gamma\in V^*$.  Since $V$ and $V^*$ are hom- left-symmetric algebras, then  we have
\[
\mathcal{K}(u,v)w=\mathcal{K}(\alpha,\beta)\gamma=0.
\]
Applying  (\ref{AM302}) and (\ref{SL}), we get
\begin{align*}\label{SS10}
&‎‎\prec \mathcal{K}(u,v)\alpha,w\succ 
=‎‎\prec L_{ { \phi}(u)}(v\cdot\alpha)-L_{{ {  { \phi}}}(v)}(u\cdot\alpha)-L_{[u,v]}{{ { {\phi^*}}}}(\alpha),w\succ =‎‎\prec {{ {  { \phi}}}(u)}\cdot(v\cdot\alpha)\\
&-{{ {  { \phi}}}(v)}\cdot(u\cdot\alpha)-{(u\cdot v)}\cdot{{{ {\phi^*}}}}(\alpha)+{(v\cdot u)}\cdot{{{ {\phi^*}}}}(\alpha),w\succ 
=‎‎\prec u\cdot({  { \phi}}(v)\cdot { {\phi^*}}(\alpha))\nonumber\\
&-v\cdot(  { \phi}(u)\cdot {\phi^*}(\alpha))-{  { \phi}(u\cdot v)}\cdot\alpha+{  { \phi}(v\cdot u)}\cdot \alpha, { \phi}(w)\succ =-‎‎\prec  { \phi}(u\cdot w), { \phi}(v)\cdot {\phi^*}(\alpha)\succ \nonumber\\
&+‎‎\prec  { \phi}(v\cdot w),  { \phi}(u)\cdot{\phi^*}(\alpha))\succ +‎‎\prec {  { \phi}(u\cdot v)}\cdot w,{\phi^*}(\alpha)\succ 
-‎‎\prec {  { \phi}(v.u)}\cdot w, {\phi^*}(\alpha)\succ \nonumber\\
&=‎‎\prec  { \phi}(v)\cdot(u\cdot w)
-  { \phi}(u)\cdot(v\cdot w)+{ (u\cdot v)}\cdot { \phi}(w)-(v\cdot u)\cdot { \phi}(w),\alpha\succ =0\nonumber,\ \ \forall w\in V,
\end{align*}
which gives $K(u,v)\alpha=0$. Similarly, we obtain $K(\alpha,\beta)u=0$.
Also  (\ref{N1}) implies
\begin{align*}
&\mathcal{K}(u,\alpha)v=L_{ { \phi}(u)}(\alpha \cdot v)-L_{{\phi^*}(\alpha)}(u\cdot v)-L_{[u,\alpha]}\phi(v)\\
&=-L_{ { \phi}(u)}L_{ {\phi^*}‎‎(\alpha)}^tv-L_{ {\phi^*}(\alpha)} (L_uv)-L_{u\cdot\alpha-\alpha.u}\phi(v)\\
&=-L_{ { \phi}(u)}L_{ {\phi^*}‎‎(\alpha)}^tv+L^t_{ ({\phi^*})‎‎^2(\alpha)} L_uv-{  { \phi}}(L^t_{L^t_{  { \phi}(u)}\alpha}v)-L_{L^t_{ {\phi^*}(\alpha)}u}{  { \phi}}(v)
={\rho}(u,\alpha)v.
\end{align*}
Similarly, we obtain $\mathcal{K}(\alpha,u)\beta=\widetilde{\rho}(\alpha,u)\beta$. 
\end{proof}
\begin{proposition}
Let $(V\oplus V^*,\cdot,\Phi)$ be an extendible hom-algebra. Then the product $\cdot$ is hom-left-symmetric if and
only if $\rho(u,\alpha)=\rho^*(\alpha,u)=0$, for any $u\in V$ and $\alpha\in V^*$.
\end{proposition} 
\begin{proof}
A direct computation yields
\begin{align*}
&((u+\alpha)\cdot(v+\beta))\cdot\Phi‎‎(w+\gamma)-\Phi‎‎(u+\alpha)\cdot((v+\beta)\cdot(w+\gamma))-((v+\beta)\cdot(u+\alpha))\cdot\Phi‎‎(w+\gamma)\\
&+\Phi‎‎(v+\beta)\cdot((u+\alpha)\cdot (w+\gamma))\\
&=\mathcal{K}(u,v)w+\mathcal{K}(u,v)\gamma+\mathcal{K}(\alpha,\beta)w+\mathcal{K}(\alpha,\beta)\gamma+{\rho^*}(\alpha,v)w-{\rho}(u,\beta)w+{\rho}^*(\alpha,v)\gamma-{\rho}(u,\beta)\gamma,
\end{align*}
for $u,v,w\in V$ and $\alpha,\beta,\gamma\in V^*$.
Using Theorem \ref{AM501}, the above equation reduces to
\begin{align*}
&((u+\alpha)\cdot(v+\beta))\cdot\Phi‎‎(w+\gamma)-\Phi‎‎(u+\alpha)\cdot((v+\beta)\cdot(w+\gamma))-((v+\beta)\cdot(u+\alpha))\cdot\Phi‎‎(w+\gamma)\\
&+\Phi‎‎(v+\beta)\cdot((u+\alpha)\cdot(w+\gamma))
={\rho^*}(\alpha,v)w-{\rho}(u,\beta)w+{\rho}^*(\alpha,v)\gamma-{\rho}(u,\beta)\gamma.
\end{align*}
The above equation gives us the assertion.
\end{proof}

\begin{proposition}
Extendible hom-algebra $(V\oplus V^*,\cdot,\Phi)$ is a
hom-Lie-admissible algebra if and only if 
\[
{\rho}(u,\alpha)v={\rho}(v,\alpha)u, \ \ \ \ \ {\rho^*}(\alpha,u)\beta={\rho^*}(\beta,u)\alpha,
\]
for any $u,v\in V$ and $\alpha,\beta\in V^*$.
\end{proposition}
\begin{proof}
If we set
\[ [u+\alpha,v+\beta]=(u+\alpha)\cdot(v+\beta)-(v+\beta)\cdot(u+\alpha),
\]
then
 (\ref{N1}) implies
\begin{align}\label{AM401}
[u+\alpha,v+\beta]=[u,v]-L^t_{{{\phi^*}}(\alpha)}v+L^t_{{\phi^*}(\beta)}u-L^t_{ { \phi}(u)}\beta+L^t_{ { \phi}(v)}\alpha+[\alpha,\beta],
\end{align}
for any $u,v\in V$ and $\alpha,\beta\in V^*$.
According to (\ref{AM502}), the product $\cdot$ is hom-Lie-admissible, if and only if   
\begin{equation}\label{SS7}
\circlearrowleft_{(u,\alpha),(v,\beta),(w,\gamma)}[\Phi(u+\alpha),[v+\beta,w+\gamma]]=0.
\end{equation}
Applying (\ref{AM401}), we obtain 
\begin{align*}
&\circlearrowleft_{(u,\alpha),(v,\beta),(w,\gamma)}[\Phi(u+\alpha),[v+\beta,w+\gamma]]
=\circlearrowleft_{(u,v,w)}[ { \phi}(u),[v,w]]+\circlearrowleft_{(\alpha,\beta\gamma)}[{\phi^*}(\alpha),[\beta,\gamma]]\\
&+\rho(u,\beta)w-\rho(u,\gamma)v-\rho(w,\beta)u+\rho(v,\gamma)u-\rho(v,\alpha)w+\rho(w,\alpha)v-\rho^*(\beta,u)\gamma+\rho^*(\gamma,u)\beta\\
&+\rho^*(\beta,w)\alpha-\rho^*(\gamma,v)\alpha+\rho^*(\alpha,v)\gamma-\rho^*(\alpha,w)\beta+\mathcal{K}(\alpha,\beta)w+\mathcal{K}(\gamma,\alpha)v+\mathcal{K}(\beta,\gamma)u\\
&+\mathcal{K}(u,v)\gamma+\mathcal{K}(w,u)\beta+\mathcal{K}(v,w)\alpha,
\end{align*}
where $u,v,w\in V$ and $\alpha, \beta, \gamma\in V^*$. Using Theorem {\ref{AM501}}, the above equation reduces to the following
\begin{align*}
&\circlearrowleft‎‎‎_{(u,\alpha),(v,\beta),(w,\gamma)}[\Phi(u+\alpha),[v+\beta,w+\gamma]]
=
\rho(u,\beta)w-\rho(u,\gamma)v-\rho(w,\beta)u+\rho(v,\gamma)u-\rho(v,\alpha)w\\
&+\rho(w,\alpha)v-\rho^*(\beta,u)\gamma+\rho^*(\gamma,u)\beta
+\rho^*(\beta,w)\alpha-\rho^*(\gamma,v)\alpha+\rho^*(\alpha,v)\gamma-\rho^*(\alpha,w)\beta,
\end{align*}
which gives the assertion.
\end{proof}
\begin{corollary}\label{SS15}
Let $(V\oplus V^*,\cdot,\Phi)$ be an extendible hom-algebra. Then there is a (involutive) hom-Lie algebra structure on $V\oplus V^*$ if and only if  for any $u,v\in V$ and $\alpha,\beta\in V^*$
\[
{\rho}(u,\alpha)v={\rho}(v,\alpha)u, \ \ \ \ \ {\rho^*}(\alpha,u)\beta={\rho^*}(\beta,u)\alpha.
\]
\end{corollary}
\begin{proof}

It is fairly easy to see that
\[
\Phi[u+\alpha,v+\beta]=[\Phi(u+\alpha),\Phi(v+\beta)].
\]
\end{proof}
\begin{theorem}
Let $T^*V=(V\oplus V^*, [\cdot,\cdot],\Phi)$ be an extendible hom-Lie algebra endowed with a non-degenerate
bilinear form $\Omega_{_{T^*V}}$ given by
\begin{equation}\label{AM505}
\Omega_{_{T^*V}}(‎ u+\alpha, v+\beta)=‎‎\prec‎ \beta,u\succ-‎‎\prec\alpha,v‎\succ,\ \ \ \ \ \forall u,v\in V, \ \ \ \forall \alpha,\beta\in V^*.
\end{equation}
Then $T^*V$ is a phase space.
\end{theorem}
\begin{proof}
Using (\ref{AM600}) and (\ref{AM401}), we get
\begin{align*}
\Omega_{_{T^*V}}(‎ \Phi(u+\alpha),\Phi( v+\beta))=&\Omega_{_{T^*V}}(‎ \phi(u)+\phi^*(\alpha),\phi( v)+\phi^*(\beta))=‎‎\prec‎ \phi^*(\beta),‎ \phi(u)\succ-‎‎\prec\phi^*(\alpha),\phi( v)‎\succ\\
&=‎‎\prec‎ \beta,u\succ-‎‎\prec\alpha,v‎\succ=\Omega_{_{T^*V}}(‎ u+\alpha, v+\beta),
\end{align*}
and 
\begin{align*}
&\circlearrowleft‎‎‎_{(u,\alpha),(v,\beta),(w,\gamma)}\Omega_{_{T^*V}}\big([‎ ‎ u+\alpha, v+\beta],\Phi(w+\gamma)\big)\\
&=\circlearrowleft‎‎‎_{(u,\alpha),(v,\beta),(w,\gamma)}\Omega_{_{T^*V}}\big([u,v]-L^t_{{{\phi^*}}(\alpha)}v+L^t_{{\phi^*}(\beta)}u-L^t_{ { \phi}(u)}\beta+L^t_{ { \phi}(v)}\alpha+[\alpha,\beta],\phi(w)+\phi^*(\gamma)\big)\\
&=\circlearrowleft‎‎‎_{(u,\alpha),(v,\beta),(w,\gamma)}\big(\prec\phi^*(\gamma),[u,v]-L^t_{{{\phi^*}}(\alpha)}v+L^t_{{\phi^*}(\beta)}u‎\succ‎‎-\prec‎-L^t_{ { \phi}(u)}\beta+L^t_{ { \phi}(v)}\alpha+[\alpha,\beta],\phi(w)\succ‎‎\big)\\
&=\prec‎\phi^*(\gamma),u\cdot v-v\cdot u+\alpha\cdot v-\beta\cdot u\succ‎‎+\prec‎-u\cdot\beta+v\cdot\alpha-\alpha\cdot\beta+\beta\cdot \alpha,\phi(w)\succ\\
&\ \ \ + \prec‎\phi^*(\alpha),v\cdot w-w\cdot v+\beta\cdot w-\gamma\cdot v\succ‎‎+\prec‎-v\cdot\gamma+w\cdot\beta-\beta\cdot\gamma+\gamma\cdot \beta,\phi(u)\succ‎‎‎‎\\
&\ \ \ + \prec‎\phi^*(\beta),w\cdot u-u\cdot w+\gamma\cdot u-\alpha\cdot w\succ‎‎+\prec‎-w\cdot\alpha+u\cdot\gamma-\gamma\cdot\alpha+\alpha\cdot \gamma,\phi(v)\succ‎‎‎‎=0,
\end{align*}
for any $u,v\in V$ and $\alpha,\beta\in V^*$. Thus $\Omega_{_{T^*V}}$ is a $2$-hom-cocycle form and consequently $T^*V$ is a phase space.
\end{proof}
\begin{lemma}
If the phase space $T^*V=(V\oplus V^*,[\cdot,\cdot],\Phi)$ is endowed with  a non degenerate
bilinear form $<,>_{_{T^*V}}$ given by
\[
<‎ u+\alpha, v+\beta>_{_{T^*V}}=‎‎\prec\alpha,v‎\succ‎‎+‎‎\prec‎ \beta,u\succ‎‎,\ \ \ \ \ \forall u,v\in V, \ \ \ \forall \alpha,\beta\in V^*.
\]
Then $(T^*V, <,>_{_{T^*V}})$ is a pseudo-Riemannian hom-Lie algebra.
\end{lemma}
\begin{proof}
Obviously, we have
\[
<‎ u+\alpha, v+\beta>_{_{T^*V}}=<‎ v+\beta,u+\alpha>_{_{T^*V}}.
\]
Also, we obtain
\begin{align*}
<\Phi(‎ u+\alpha), \Phi(v+\beta)>_{_{T^*V}}=&<‎ \phi(u)+\phi^*(\alpha),\phi( v)+\phi^*(\beta)>_{_{T^*V}}
=‎‎\prec\phi^*(\alpha),\phi( v)\succ+‎‎\prec‎ \phi^*(\beta),‎ \phi(u)\succ\\
&=‎‎\prec\alpha,v‎\succ‎‎+‎‎\prec‎ \beta,u\succ=<u+\alpha,v+\beta>_{_{T^*V}},
\end{align*}
for any $u,v\in V$ and $\alpha,\beta\in V^*$.
\end{proof}
Let  $K:V\oplus V^*\rightarrow V\oplus V^*$ be the endomorphism given by $K(u+\alpha)=\phi(u)-\phi^*(\alpha)$. Then it follows that $(\Phi\circ K)(u+\alpha)=u-\alpha$,  $K\circ\Phi=\Phi\circ K$ and $K^2=Id$. Moreover
\[
(\Phi\circ K)^2(u+\alpha)=(\Phi\circ K)(u-\alpha)=u+\alpha.
\]
Also, we obtain
\begin{align*}
<(\Phi\circ K)(u+\alpha),v+\beta>_{_{T^*V}}=&<u-\alpha,v+\beta>=-‎‎\prec\alpha,v‎\succ‎‎+‎‎\prec‎ \beta,u\succ\\
=&-<u+\alpha,(\Phi\circ K)(v+\beta)>_{_{T^*V}}.
\end{align*}
On the other hand, we get $L_{u+\alpha}\circ\Phi\circ K=\Phi\circ K\circ L_{u+\alpha}$, because
\begin{align*}
(L_{u+\alpha}\circ\Phi\circ K)(v+\beta)=&L_{u+\alpha}(v-\beta)=L_{u+\alpha}(v-\beta)=(u+\alpha)\cdot(v-\beta)\\
=&u\cdot v-L^t_{{{\phi^*}}(\alpha)}v+L^t_{  { \phi}(u)}\beta-\alpha\cdot\beta,
\end{align*}
and
\begin{align*}
(\Phi\circ K\circ L_{u+\alpha})(v+\beta)=&(\Phi\circ K)(({u+\alpha})\cdot(v+\beta))=(\Phi\circ K)(u\cdot v-L^t_{{{\phi^*}}(\alpha)}v-L^t_{  { \phi}(u)}\beta+\alpha\cdot\beta)\\
=&u\cdot v-L^t_{{{\phi^*}}(\alpha)}v+L^t_{  { \phi}(u)}\beta-\alpha\cdot\beta.
\end{align*}
From the above expression, we can deduce the following
\begin{theorem}
Let $(V\oplus V^*,\cdot,\Phi)$ be an extendible hom-algebra. Then $(T^*V, <,>_{_{T^*V}},K)$ is a para-K\"{a}hler hom-Lie algebra. 
\end{theorem}

\end{document}